\documentclass[a4paper, 11pt]{article}

\usepackage[bb=boondox]{mathalfa}
\usepackage{amsmath,amsthm,amsfonts,amssymb,bbm,mathtools}
\usepackage[margin=2cm]{geometry}
\usepackage{etoolbox}
\usepackage[dvipsnames]{xcolor}
\usepackage{svg}
\usepackage{tikz}
\usepackage{comment}
\usepackage{url}

\usepackage{cleveref}

\usepackage[shortlabels]{enumitem}

\theoremstyle{plain}
\newtheorem{theorem}{Theorem}[section]

\newtheorem{corollary}  [theorem]{Corollary}

\newtheorem{example}    [theorem]{Example}
\newtheorem{lemma}      [theorem]{Lemma}

\newtheorem{proposition}[theorem]{Proposition}
\newtheorem{remark}     [theorem]{Remark}

\theoremstyle{definition}
\newtheorem{claim}  {Claim}
\newtheorem*{claim*}{Claim}

\newcommand{\N} {\mathbb{N}}

\newcommand{\F} {\mathbb{F}}
\renewcommand{\S}{\mathbb{S}}

\newcommand{\Define}    [1] {\textbf{#1}}

\newcommand{\join}      {\lor}
\newcommand{\meet}      {\land}
\newcommand{\bigjoin}   {\bigvee}
\newcommand{\bigmeet}   {\bigwedge}
\newcommand{\zero}      {\mathbf{0}}
\newcommand{\one}       {\mathbf{1}}
\newcommand{\Atoms}     {\mathrm{A}}
\newcommand{\lcm}       {\mathrm{lcm}}
\newcommand{\Height}    {\mathrm{h}}

\title{Semirings of formal sums and injective partial transformations}
\author{Maximilien Gadouleau\footnote{Department of Computer Science, Durham University, UK. \texttt{m.r.gadouleau@durham.ac.uk}} \and Marianne Johnson\footnote{Department of Mathematics,  University of Manchester, UK. \texttt{marianne.johnson@mathematics.manchester.ac.uk}}}
\date{\today}

\begin{document}

\maketitle

\begin{abstract}
The semiring of discrete dynamical systems is a simple algebraic model for modularity in deterministic systems. The objects of the semiring are finite transformations (viewed as directed graphs and regarded up to isomorphism), the sum of two transformations corresponds to applying them independently on distinct sets, and the product corresponds to applying both transformations in parallel. In this paper, we extend this semiring to include partial transformations; the sum and product are natural generalisations. Each (partial) transformation can be viewed as a sum (over $\N$) of connected (partial) transformations. We generalise this idea by working in semirings of formal sums over any semiring $\mathbb{S}$. Here we consider the case where $\mathbb{S} = \F_2$, the binary field, and we focus on injective partial transformations, i.e. sums of chains and cycles. While no efficient algorithm for the division problem for sums of cycles in the original semiring of discrete dynamical systems is known, we give a concise characterisation of all the solutions of the division problem for sums of cycles over $\F_2$. We then extend this characterisation to dividing any injective partial transformations, i.e. sums of chains and cycles over $\F_2$.
\end{abstract}

\section{Introduction} \label{section:introduction}

\subsection{The semiring of transformations} \label{subsection:semiring_of_transformations}

Many large systems, both abstract and in real life, are modular, i.e. they are built by combining smaller modules that perform individual tasks. Some examples include Boolean networks \cite{PPS21, KWVML23}, systolic arrays \cite{Pet92}, Cellular Automata \cite{Ped92, CCG01, McI09}, and Petri nets \cite{Dev16, Dev21, TCV22}. Similarly, in algebra, the direct product of structures (algebraic or relational) is an elementary operation, that can be defined over any variety, such as groups, rings, or semigroups; while in graph theory, the disjoint union is a basic way of combining two digraphs.

The semiring of transformations was first proposed in 2018 by researchers in Cellular Automata \cite{DDFMP18} as a simple mathematical model for deterministic systems.\footnote{It has appeared under the name of semiring of Discrete Dynamical Systems, Finite Dynamical Systems, or functional digraphs; we shall refer to those objects as transformations.} The objects of the semiring are finite transformations (functions from a finite set to itself) up to graph isomorphism. Those objects are equipped with two operations: the sum corresponds to two systems having disjoint dynamics, while the product corresponds to the parallel composition of the two dynamics.

The proposed model aims to be very broad by remaining simple: the modelled system has a finite number of possible states, for which we do not know the semantics, and is represented by a deterministic evolution. The two operations represent different ways of combining those systems. While (full) transformations assume that every state has a defined successor, some states may lead to undefined characteristics, or crashing the whole system entirely; such states without successors are encapsulated in a partial transformation instead. This importantly also applies to the case where only partial information on the transitions are available. As such, in this paper we extend the semiring of transformations to the semiring of partial transformations; the sum and product are defined naturally in this case.

An example of a sum and product of two partial transformations is given in Figure \ref{figure:one_example}.  The objects of the semiring are the unlabelled graphs, displayed on the top row. The bottom row displaying the labels illustrates how the product works: for instance, the arcs $\textcolor{red}{b} \to \textcolor{red}{a}$ in $\textcolor{red}{A}$ and $\textcolor{blue}{3} \to \textcolor{blue}{1}$ in $\textcolor{blue}{B}$ yield the arc $\textcolor{orange}{b3} \to \textcolor{orange}{a1}$ in $\textcolor{orange}{AB}$. Note that the product is commutative. Partial transformations may be represented by digraphs where each vertex has a unique successor. In graph terms, the sum represents the disjoint union of digraphs, while the product is the direct product of digraphs \cite{HIK11}.

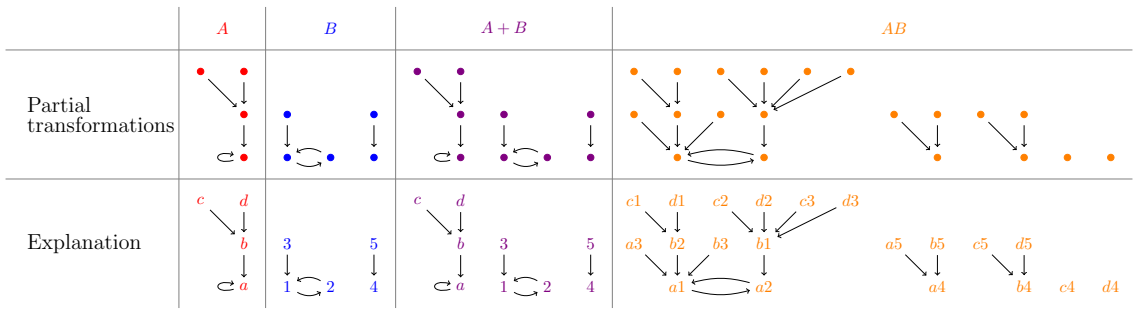
\begin{figure}[bp]
\centering
\resizebox{!}{4cm}{
\begin{tikzpicture}

    \draw[gray] (0.5,-3.5) -- (0.5,3.5);
    \draw[gray] (3.5,-3.5) -- (3.5,3.5);
    \draw[gray] (8.5,-3.5) -- (8.5,3.5);


    \draw[gray] (-5.5,2.5) -- (20.5,2.5);
    \draw[gray] (-5.5,-0.5) -- (20.5,-0.5);
    \draw[gray] (-1.5,-3.5) -- (-1.5,3.5);

    \node[text width = 4cm] (p) at (-3,1) {\Large Partial \newline transformations};
    \node[text width = 4cm] (e) at (-3,-2) {\Large Explanation};

    \begin{scope}[xshift = 0cm, yshift = -3cm, every node/.style={text=red}]
        
        \node (a) at (0,0) {$a$};  
        \node (b) at (0,1) {$b$};
        \node (c) at (-1,2) {$c$};
        \node (d) at (0,2) {$d$};

        \draw[->] (d) -- (b);
        \draw[->] (c) -- (b);
        \draw[->] (b) -- (a);
        \draw[->] (a) to [loop left] (a);
    \end{scope}

    \begin{scope}[xshift = 0cm, yshift = 0cm, vertex/.style={circle, draw=red, fill=red, inner sep=1.5, outer sep = 5}]
        \node[text = red] (A) at (-0.5,3) {$A$};
        \node[vertex] (a) at (0,0) {};  
        \node[vertex] (b) at (0,1) {};
        \node[vertex] (c) at (-1,2) {};
        \node[vertex] (d) at (0,2) {};

        \draw[->] (d) -- (b);
        \draw[->] (c) -- (b);
        \draw[->] (b) -- (a);
        \draw[->] (a) to [loop left] (a);
        
    \end{scope}
    
    \begin{scope}[xshift = 1cm, yshift = -3cm, every node/.style={text=blue}]
        \node (1) at (0,0) {$1$};  
        \node (2) at (1,0) {$2$};
        \node (3) at (0,1) {$3$};

        \node (4) at (2,0) {$4$};  
        \node (5) at (2,1) {$5$};

        \draw[->] (3) -- (1);
        \draw[->] (2) to [bend right] (1);
        \draw[->] (1) to [bend right] (2);

        \draw[->] (5) -- (4);
    \end{scope}

    \begin{scope}[xshift = 1cm, yshift = 0cm, vertex/.style={circle, draw=blue, fill=blue, inner sep=1.5, outer sep = 5}]
        \node[text = blue] (B) at (1,3) {$B$};
        \node[vertex] (1) at (0,0) {};  
        \node[vertex] (2) at (1,0) {};
        \node[vertex] (3) at (0,1) {};

        \node[vertex] (4) at (2,0) {};  
        \node[vertex] (5) at (2,1) {};

        \draw[->] (3) -- (1);
        \draw[->] (2) to [bend right] (1);
        \draw[->] (1) to [bend right] (2);

        \draw[->] (5) -- (4);
    
    \end{scope}

    \begin{scope}[xshift = 5cm, yshift = -3cm, every node/.style={text=violet}]
        \node (a) at (0,0) {$a$};  
        \node (b) at (0,1) {$b$};
        \node (c) at (-1,2) {$c$};
        \node (d) at (0,2) {$d$};

        \draw[->] (d) -- (b);
        \draw[->] (c) -- (b);
        \draw[->] (b) -- (a);
        \draw[->] (a) to [loop left] (a);

        \node (1) at (1,0) {$1$};  
        \node (2) at (2,0) {$2$};
        \node (3) at (1,1) {$3$};

        \node (4) at (3,0) {$4$};  
        \node (5) at (3,1) {$5$};

        \draw[->] (3) -- (1);
        \draw[->] (2) to [bend right] (1);
        \draw[->] (1) to [bend right] (2);

        \draw[->] (5) -- (4);
    \end{scope}

    \begin{scope}[xshift = 5cm, yshift = 0cm, vertex/.style={circle, draw=violet, fill=violet, inner sep=1.5, outer sep = 5}]]
        \node[text = violet] (A+B) at (1,3) {$A+B$};
        \node[vertex] (a) at (0,0) {};  
        \node[vertex] (b) at (0,1) {};
        \node[vertex] (c) at (-1,2) {};
        \node[vertex] (d) at (0,2) {};

        \draw[->] (d) -- (b);
        \draw[->] (c) -- (b);
        \draw[->] (b) -- (a);
        \draw[->] (a) to [loop left] (a);

        \node[vertex] (1) at (1,0) {};  
        \node[vertex] (2) at (2,0) {};
        \node[vertex] (3) at (1,1) {};

        \node[vertex] (4) at (3,0) {};  
        \node[vertex] (5) at (3,1) {};

        \draw[->] (3) -- (1);
        \draw[->] (2) to [bend right] (1);
        \draw[->] (1) to [bend right] (2);

        \draw[->] (5) -- (4);
    \end{scope}

    \begin{scope}[xshift = 10cm, yshift = -3cm, every node/.style={text=orange}]
        \node (a1) at (0,0) {$a1$};  
        \node (b1) at (2,1) {$b1$};
        \node (c1) at (-1,2) {$c1$};
        \node (d1) at (0,2) {$d1$};

        \node (a2) at (2,0) {$a2$};  
        \node (b2) at (0,1) {$b2$};
        \node (c2) at (1,2) {$c2$};
        \node (d2) at (2,2) {$d2$};

        \node (a3) at (-1,1) {$a3$};  
        \node (b3) at (1,1) {$b3$};
        \node (c3) at (3,2) {$c3$};
        \node (d3) at (4,2) {$d3$};

        \node (a4) at (6,0) {$a4$};  
        \node (b4) at (8,0) {$b4$};
        \node (c4) at (9,0) {$c4$};
        \node (d4) at (10,0) {$d4$};

        \node (a5) at (5,1) {$a5$};  
        \node (b5) at (6,1) {$b5$};
        \node (c5) at (7,1) {$c5$};
        \node (d5) at (8,1) {$d5$};

        \draw[->] (d1) -- (b2);
        \draw[->] (c1) -- (b2);
        \draw[->] (b1) -- (a2);
        \draw[->] (a1) to [bend right=15] (a2);

        \draw[->] (d2) -- (b1);
        \draw[->] (c2) -- (b1);
        \draw[->] (b2) -- (a1);
        \draw[->] (a2) to [bend right=15] (a1);

        \draw[->] (d3) -- (b1);
        \draw[->] (c3) -- (b1);
        \draw[->] (b3) -- (a1);
        \draw[->] (a3) -- (a1);

        \draw[->] (d5) -- (b4);
        \draw[->] (c5) -- (b4);
        \draw[->] (b5) -- (a4);
        \draw[->] (a5) -- (a4);

    \end{scope}

    \begin{scope}[xshift = 10cm, yshift = 0cm, vertex/.style={circle, draw=orange, fill=orange, inner sep=1.5, outer sep = 5}]]
        \node[text = orange] (AxB) at (5,3) {$AB$};
        \node[vertex] (a1) at (0,0) {};  
        \node[vertex] (b1) at (2,1) {};
        \node[vertex] (c1) at (-1,2) {};
        \node[vertex] (d1) at (0,2) {};

        \node[vertex] (a2) at (2,0) {};  
        \node[vertex] (b2) at (0,1) {};
        \node[vertex] (c2) at (1,2) {};
        \node[vertex] (d2) at (2,2) {};

        \node[vertex] (a3) at (-1,1) {};  
        \node[vertex] (b3) at (1,1) {};
        \node[vertex] (c3) at (3,2) {};
        \node[vertex] (d3) at (4,2) {};

        \node[vertex] (a4) at (6,0) {};  
        \node[vertex] (b4) at (8,0) {};
        \node[vertex] (c4) at (9,0) {};
        \node[vertex] (d4) at (10,0) {};

        \node[vertex] (a5) at (5,1) {};  
        \node[vertex] (b5) at (6,1) {};
        \node[vertex] (c5) at (7,1) {};
        \node[vertex] (d5) at (8,1) {};

        \draw[->] (d1) -- (b2);
        \draw[->] (c1) -- (b2);
        \draw[->] (b1) -- (a2);
        \draw[->] (a1) to [bend right=15] (a2);

        \draw[->] (d2) -- (b1);
        \draw[->] (c2) -- (b1);
        \draw[->] (b2) -- (a1);
        \draw[->] (a2) to [bend right=15] (a1);

        \draw[->] (d3) -- (b1);
        \draw[->] (c3) -- (b1);
        \draw[->] (b3) -- (a1);
        \draw[->] (a3) -- (a1);

        \draw[->] (d5) -- (b4);
        \draw[->] (c5) -- (b4);
        \draw[->] (b5) -- (a4);
        \draw[->] (a5) -- (a4);
    \end{scope}

\end{tikzpicture}
}
\caption{Sum and product of partial transformations.}
\label{figure:one_example}
\end{figure}

The semiring of transformations was introduced in \cite{DDFMP18} with a focus on algorithmic problems. For instance, algorithms to solve simple equations (related to the division equation $ax = b$) are given in \cite{DFMMR19, FRR21, BCDR24, DFMR24} and an efficient algorithm to enumerate transformations is given in \cite{DPT24}. Moreover, a natural framework for a class of reductions, called abstractions, is given in \cite{Riv22}. Concretely, the idea is to focus on a particular aspect of transformations in order to simplify solving equations. Natural examples of reductions include the cardinality, the transient parts, the periodic parts and homomorphic images. The so-called profile, introduced in \cite{Gaz20}, counts states according to their distance to the periodic part. The unroll \cite{NG24} is a homomorphism to the semiring of (possibly infinite) forests, which preserves the transient part. The unroll became the key to understand the transient part in the semiring of transformations, and as such has been used to classify the cancellable transformations ($a$ such that $ax = ay$ implies $x = y$) in \cite{NG24} and more generally to classify the injective univariate polynomials \cite{PR25CiE}.

Solving a multivariate equation $P(\vec{x}) = c$, where $P$ is a multivariate polynomial over the semiring of transformations and $c$ is a fixed transformation, is believed to be an NP-complete problem, since the counterpart problem for profiles is indeed NP-complete \cite{Gaz20}. However, exhibiting algebraic properties, such as the unroll homomorphism or the classification of injective univariate polynomials, tends to have algorithmic consequences as well. For instance, \cite{NG24} gives a polynomial-time algorithm to divide two dendrons (transformations where all trajectories eventually lead to a unique fixed point) while \cite{PR25CiE} gives a polynomial-time algorithm to solve the equation $P( x ) = c$ for an injective univariate polynomial $P$ and a fixed transformation $c$; the latter has been further generalised in \cite{PR25A} to so-called pseudo-injective polynomials.

Nonetheless, the complexity of many computational problems in the semiring of transformations remains unknown. A key problem is the division equation $ax = b$, even when restricted to permutations, i.e. sums of cycles. The current fastest algorithm to solve $ax = b$ when $a$ and $b$ are permutations, given in \cite{BCDR24}, runs in sub-exponential time.

\subsection{Contributions and outline} \label{subsection:contributions}

A partial transformation of a set $X$ is a mapping $f : Y \to X$, where $Y \subseteq X$. From $f$ one can define the directed graph $D_f$ with vertex set $V(D_f) = X$ and arc set $E(D_f) = \{ (y, f(y)) : y \in Y \}$. In this digraph, every vertex has out-degree at most one; conversely, any digraph where every vertex has out-degree at most one can be constructed that way. We say that two partial transformations $f$ and $f'$ of $X$ and $X'$ are equivalent if $D_f$ and $D_{f'}$ are isomorphic. We then consider classes of equivalence of partial transformations. The sum of two (equivalence classes of) partial transformations $f : Y \to X$ and $f' : Y' \to X'$ with $X$ and $X'$ disjoint is their disjoint union: $f+f' : Y \cup Y' \to X \cup X'$ with $(f + f')(y) = f(y)$ for all $y \in Y$ and $(f + f')( y' ) = f'( y' )$ for all $y' \in Y'$. Their product is given by the direct product of digraphs: $f f' : Y \times Y' \to X \times X'$ with $(ff')( y, y' ) = (f(y), f(y'))$ for all $(y, y') \in Y \times Y'$. It is clear that those operations are well-defined, since sum and product preserve the equivalence of partial transformations.

Our paper makes two main qualitative contributions. The first qualitative contribution is to consider partial transformations. We motivate this generalisation in two ways. Firstly, in many systems, some configurations lead to an undefined state, or a failure of some kind. As such, all states may not always have a successor, and the system is represented by a partial transformation. Secondly, in system biology, one learns the system through discovering some of the transitions in a successive manner, and some transitions may be impractical to actually discover. As such, the current knowledge of the system must be described by a partial transformation. Partial transformations form a semiring, by naturally extending the definitions of sum and product given for transformations.

The second main qualitative contribution is to introduce torsion to the semiring of partial transformations, by considering formal sums over a finite semiring $\mathbb{S}$. This immediately brings additive torsion, but can also bring multiplicative torsion as well. In the original semiring of transformations, elements have no torsion, since the vertex sets of $a, a^2, a^3, \dots$ all have distinct sizes; similarly $a, 2a, 3a$, and so on are all distinct. For instance, the cycle $C_d$ satisfies $C_d^2 = d C_d$, and hence $C_d^k = d^{k-1} C_d$ for all $k \ge 1$. On the other hand, if for instance $\mathbb{S} = \F_2$ is the binary field, then we count sums of cycles ``modulo $2$'', e.g. the permutation $C_3 + 4C_4 + 2C_5 + 7C_8$ becomes $C_3 + C_8$. In this semiring, we have $C_d^2 = 0$ if $d$ is even and $C_d^2 = C_d$ if $d$ is odd; and moreover, $C_d + C_d = 0$. These relations not only greatly simplify calculations, but also yield a rich algebraic structure which can be exploited (see Section \ref{section:semiring_C2}).

In this paper we consider the division problem  (given $a, b \in S$ determine whether there exists a solution to $ax=b$) in certain commutative semirings, which is equivalent to studying (any one of) Green's relations in the underlying multiplicative semigroup. Particularly we consider the subsemiring of equivalence classes of injective partial transformations with finite support, and each such graph can be decomposed into cycles and chains. More concretely, for each positive integer $d$, let $L_d = d \to (d-1) \to \dots \to 1$ be the \Define{chain} of length $d$ and as before $C_d = d \to d-1 \to \dots \to 1 \to d$ be the \Define{cycle} of length $d$. We denote the set of all finite chains as $L$ and the set of all finite cycles as $C$. Our main technical contributions centre around the division problem for sums of cycles, when counting modulo $2$, i.e. working in $\mathbb{S}C$, where $\mathbb{S} = \mathbb{F}_2$ and $C = \{C_d : d \ge 1\}$ is the set of all cycles. We show in Theorem \ref{theorem:ax=b} that the solutions to $ax = b$ when counting modulo $2$ can be expressed in terms of intervals in a Boolean algebra, and in particular deciding whether $ax = b$ has a solution can be done efficiently.  We extend this result in two ways. Firstly, reflecting on the work in \cite{PR25CiE} on injective polynomials, we classify the injective univariate polynomials over $\F_2 C$ and show that solving equations for those polynomials is straightforward in Theorem \ref{theorem:injective_polynomials}. Secondly, we solve the division problem for injective partial transformations, i.e. in $\F_2(L \cup C)$, in Theorem \ref{theorem:division_LC}.

The rest of the paper is organised as follows. Firstly, we introduce the general framework of semirings of formal sums in Section \ref{subsection:semirings_general}, and we focus on the important case of formal sums of elements of a semilattice over $\F_2$ in Section \ref{subsection:semilattice_F2}. Secondly, we carry out a thorough investigation of the semiring $\F_2 C$ of formal sums of cycles modulo $2$ in Section \ref{section:semiring_C2}: we give some general basic results in Section \ref{subsection:S_basic_properties}; we characterise the Boolean algebra of idempotents in Section \ref{subsection:boolean_algebra_S0}; we provide a characterisation of the solutions to the division equation $ax = b$ in Section \ref{subsection:division_in_S}; we characterise the (extended) Green's relations in that semigroup in Section \ref{subsection:Green}; we classify units, regular elements and more in Section \ref{subsection:special_elements}; we classify injective univariate polynomials in Section \ref{subsection:injective_polynomials}; and we study the intersection of principal ideals in Section \ref{subsection:intersection_principal_ideals}. Finally, we consider the division problem in the semiring of formal sums of injective partial transformations in Section \ref{section:injective}: we first consider the easy case of nilpotent injective partial transformations (i.e. sums of chains) in Section \ref{subsection:semiring_L2} and then the general case in Section \ref{subsection:division_LC}. The paper ends in the Conclusion Section \ref{section:conclusion}, where we give some possible avenues for future work.

\section{Semirings of formal sums}
\label{section:semiring_formal_sums}

\subsection{General case} \label{subsection:semirings_general}

A commutative semiring is a structure $(\mathbb{S}, +, \times, 0, 1)$, where $(\mathbb{S}, +)$ is a commutative monoid with identity $0$, $(\mathbb{S}, \times)$ is a commutative monoid with identity $1$, multiplication distributes over addition, i.e. $a \times ( b + c ) = ( a \times b ) + (a \times c )$ for all $a, b, c \in \mathbb{S}$, and $0$ is a multiplicative zero, i.e. $0 \times a = a \times 0 = 0$ for all $a \in \mathbb{S}$. Since all semirings considered in this paper are commutative, we will frequently say simply `semiring'. As usual, when the context is clear, we shall write $ab$ in place of $a \times b$. We denote the semiring of natural numbers with respect to the usual arithmetic operations by $\N = \{0, 1, \dots, \}$.

Let $G$ be a set and $\mathbb{S}$ be a semiring. Let $\mathbb{S}G$ be the set of formal sums of elements in $G$ with coefficients in $\mathbb{S}$ and finite support, i.e.
\[
    \mathbb{S}G = \left\{ \sum_{j \in G} s_j j : s_j \in \mathbb{S}, \text{ finitely many non-zero } s_j \right \}.
\]
For $s \in \mathbb{S}G$  and $j \in G$ we will write $s_j$ to denote the coefficient of $j$ in $s$. Notice that there is an obvious left action of $\mathbb{S}$ on $\mathbb{S}G$ determined by $(as)_j = as_j$ for all $a \in \mathbb{S}, s \in \mathbb{S}G, j\in G$, and a point-wise addition on $\mathbb{S}$ determined by $(s+t)_j = s_j + t_j$ for all $s,t \in \mathbb{S}G$ and all $j \in G$. It is clear that this operation is associative and commutative. For $g \in G$ we shall also (by a slight abuse of notation) write simply $g$ for the element $1g$ of $\mathbb{S} G$. Let us also write $0$ for the empty sum, which is the additive identity element. In order to specify a semiring structure on $\mathbb{S}G$, it suffices to specify a function of the form $G \times G \rightarrow \mathbb{S}G$ obeying certain natural conditions.
 
\begin{theorem}\label{theorem:NG_semiring}
Let $\mathbb{S}$ be a commutative semiring and suppose that $G$ is a set endowed with an operation $\cdot : G \times G \to \mathbb{S} G$, satisfying the following three conditions:
\begin{enumerate}
\item for all $g,h \in G$,  $g\cdot h = h\cdot g$;
\item there exists $e \in G$ such that $e \cdot g = g \cdot e=g \in \mathbb{S} G$ for all $g \in G$;
\item for all $g, h, i, k \in G$
\[
    \sum_{j \in G} (g\cdot h)_j (j \cdot i)_k = \sum_{j \in G} (h\cdot i)_j (j\cdot g)_k.
\]
\end{enumerate}Then $\mathbb{S} G$  is a semiring with respect to point-wise addition and multiplication given by:
\begin{equation}
\label{eq:prod}
\sum_{g \in G} s_g g \cdot \sum_{h \in G} t_h h = \sum_{g,h \in G} s_g t_h (g \cdot h).
\end{equation}
Furthermore, if $\mathbb{S}$ is a ring, then $\mathbb{S}G$ is a ring.
\end{theorem}
\begin{proof}
 First note that the formula in \eqref{eq:prod} determines  well-defined function $\mathbb{S} G \times \mathbb{S}G \rightarrow \mathbb{S}G$: by definition $g \cdot h$ is a formal sum with coefficients in $\mathbb{S}$ and finite support; since $\mathbb{S}$ acts on the left of  $\mathbb{S} G$ it follows that for all $g,h \in G$ and all $a,b \in \mathbb{S}$ we have $ab (g \cdot h) \in \mathbb{S} G$; then the fact that only finitely many of the coefficients $s_g$ and $t_h$ are non-zero yields that the above element has finite support. It is also easy to see that the empty sum $0$ is a left and right zero element with respect to this multiplication. Condition 1 together with the fact that multiplication in $\mathbb{S}$ is commutative clearly implies that this product is commutative, whilst the existence of the element $e$ from condition 2 ensures that $e \in \mathbb{S} G$ is an identity element for the multiplication. Next note that for all $g,h,i \in G$ we have
\begin{eqnarray*}
    (g\cdot h)\cdot i &=& \left(\sum_{j \in G}  (g\cdot h)_j j\right) \cdot i = \sum_{j \in G} ((g\cdot h)_j (j \cdot i)) =\sum_{j \in G} \left((g\cdot h)_j  \sum_{k \in G} (j\cdot i)_k  k\right) = \sum_{k \in G} \sum_{j \in G} (g\cdot h)_j (j\cdot i)_k  k\\
g\cdot (h\cdot i) &=& g \cdot \left(\sum_{j \in G}  (h\cdot i)_j j\right) = \sum_{j \in G}  (h\cdot i)_j (g \cdot j) =\sum_{j \in G} \left((h\cdot i)_j  \sum_{k \in G} (g\cdot j)_k  k\right) = \sum_{k \in G} \sum_{j \in G} (h\cdot i)_j (g\cdot j)_k  k,
\end{eqnarray*}
so we have $(g \cdot h) \cdot i = g \cdot (h \cdot i)$ if and only if condition 3 holds. From this it is then routine to check that the multiplication defined on $\mathbb{S} G$ is associative, and moreover that the multiplication distributes over the addition. Thus $\mathbb{S} G$ is a semiring.
\end{proof}

\begin{example}\label{example:semigroup_rings}
If $G$ is a commutative monoid and $\mathbb{S}$ a semiring, by taking $g \cdot h = gh \in \mathbb{S}G$ it is easy to see that each of the conditions above are satisfied. Indeed,  associativity of the monoid product in $G$ ensures that the third condition is satisfied, since $(ab)_c = \delta( ab, c )$ for all $a, b, c \in G$ and hence
\[
    \sum_j (gh)_j (ji)_k = \delta( (gh)i, k ) = \delta( g(hi), k ) = \sum_j (hi)_j (gj)_k.
\]
In this case $\mathbb{S}G$ with respect to the given operations is the monoid algebra of $G$ over $\mathbb{S}$ (typically studied in the case where $\mathbb{S}$ is taken to be a field).     
\end{example}

\begin{remark}\label{remark:SG_semiring}
Suppose we are given an operation $G\times G \rightarrow \N G$ satisfying the conditions of Theorem \ref{theorem:NG_semiring}. Then, for any semiring $\mathbb{S}$  we may define a multiplication on $\mathbb{S}G$ using the operation $\cdot: G \times G \rightarrow \N G$. Specifically if $\phi$ is the unique semiring homomorphism from $\N$ to $\mathbb{S}$ we may define $\cdot_\mathbb{S} : G \times G \to \mathbb{S}G$ as $g \cdot_\mathbb{S} h = \sum_j \phi( (g\cdot h)_j ) j$, and it is easy to check that this map also satisfies the conditions of Theorem \ref{theorem:NG_semiring}.
\end{remark}

\begin{example}
Taking $G$ to be the set of all finite directed graphs with maximum out-degree at most one, regarded up to isomorphism, the semiring of discrete dynamical systems is then the semiring $\N G$ where the product of two graphs is determined by a function $G \times G \rightarrow \N G$ as described in the introduction. For example, with $A$ and $B$ as in Figure \ref{figure:one_example} we find that $AB = M + 2N + 2 P$ where $M$ is the component with 12 vertices, $N$ is the component (occurring twice up to isomorphism) with three vertices and $P$ is the component (also occurring twice up to isomorphism) consisting of a single vertex. For $\mathbb{S} = \mathbb{F}_2$ the field of two elements, we may also consider the ring $\mathbb{F}_2 G$, which corresponds to considering all coefficients in $\N$ modulo $2$. For instance, in the example above, we have $AB = M$ when regarded as elements of $\mathbb{F}_2 G$.  In the following sections we shall be interested in certain subrings of $\mathbb{F}_2 G$.
\end{example}

\subsection{Formal sums over finite semilattices with $\F_2$} \label{subsection:semilattice_F2}

We now focus on the particular case where $G$ is a finite semilattice and $\mathbb{S} = \F_2$, which will be useful for Sections \ref{section:semiring_C2} and \ref{section:injective}. In this case, the structure of $\mathbb{S}G$ is that of a Boolean algebra.

Let us review some basic notation and terminology for lattices; see the book by Gr\"azer \cite{Gra03} for an authoritative review of lattices. Let $\mathcal{L} = ( L, \join, \meet, \zero, \one )$ be a bounded lattice with top element denoted by $\one$ and bottom element denoted by $\zero$. We naturally denote $l \le m$ for $l = l \meet m$, or equivalently $m = l \join m$. For any $l \in L$, we denote $l^\uparrow = \{ m \in L : m \ge l \}$, $l^+ = \{ m : m > l \}$, and $l^\downarrow = \{ m \in L : m \le l \}$. Note that $l^\downarrow$ is also a bounded lattice with top $l$ and bottom $0$. For any $l \in L$, let $l^c$ be the set of elements of $L$ \Define{covered} by $l$, i.e. $l^c = \{ m \in L : m < l, \not\exists n : m < n < l \}$. An element $a \in L$ is an \Define{atom} of $\mathcal{L}$ if $a^c = \{ \zero \}$; we denote the set of atoms of $\mathcal{L}$ as $\Atoms( \mathcal{L} )$. Recall that for all $l, m \in L$, the \Define{interval} $[l,m]$ is given by $[l,m] = \{ x \in L : l \le x \le m \}$; this interval is empty unless $l \le m$.

Let $\mathcal{L} = (L, \join, \meet, \zero, \one)$ be a finite lattice on $n$ elements. Then 
\[
    \F_2 L = \left\{ \sum_{ l \in L } a_l l : a_l \in \F_2 \ \forall l \in L \right\},
\]
endowed with the addition and product from Example \ref{example:semigroup_rings}, treating $L$ as a monoid with respect to meet, forms a Boolean ring. In order to emphasise the distinction between the element $l \in L$ and the formal sum $1 \cdot l$, we shall refer to the latter as the unit vector $e_l$. In turn, the Boolean ring on $\F_2 L$ induces the Boolean algebra $V( \mathcal{L} ) = ( \F_2 L , \join_{ V( \mathcal{L} ) }, \meet_{ V( \mathcal{L} ) }, \zero_{ V( \mathcal{L} ) }, \one_{ V( \mathcal{L} ) } )$, where:
\begin{itemize}
    \item the meet is given by $a \meet_{ V( \mathcal{L} ) } b = ab$,

    \item the join is given by $a \join_{ V( \mathcal{L} ) } b = a + b + ab$,

    \item the lower bound is the all-zero vector (empty sum) $\zero_{ V( \mathcal{L} ) } = 0$,

    \item the upper bound is the unit vector $\one_{ V( \mathcal{L} ) } = e_\one$,

    \item the complement is given by $\neg_{ V( \mathcal{L} ) } a = a + e_\one$.
\end{itemize}
We shall omit the subscript $_{ V( \mathcal{L} ) }$ from our notation when it is clear from the context that we are working in $V( \mathcal{L} )$.

All finite lattices of the same order yield isomorphic Boolean algebras: if $|L| = n$, then $V( \mathcal{L} )$ is isomorphic to the power set lattice of $\{1, \dots, n\}$, which of course has $n$ atoms \cite{GH09}. However, $V( \mathcal{L} )$ does encapsulate the structure of $\mathcal{L}$, since $e_l \le e_m$ in $V( \mathcal{L} )$ if and only if $l \le m$ in $\mathcal{L}$.

Each element of $\F_2 L$ may be identified with its support, that is, the subset of elements of $L$ arising with non-zero coefficient. As such, if $a \in \F_2 L$, we write $l \in a$ if $a_l = 1$, and we denote $\bigjoin a = \bigjoin_{l \in a} l$; addition $a+b$ in $V( \mathcal{L} )$ then corresponds to the symmetric difference of sets $a \Delta b$. 

In Section \ref{subsection:boolean_algebra_S0}, we shall work in the Boolean algebras $\F_2 L$ for a specific kind of finite lattices arising naturally from odd cycles. It will be very useful to characterise the atoms of those algebras. More generally, in order to perform computations in $\F_2 L$, determining its atoms is a basic requirement. As such, we construct in Theorem \ref{theorem:atoms} below all the atoms of $\F_2 L$. We begin with a construction.

\begin{lemma} \label{lemma:a_star}
Let $\mathcal{D} = (D, \vee, \meet, \one_\mathcal{D}, \one_\mathcal{D})$ be a finite lattice. Then there is a unique subset $a^*( \mathcal{D} ) \subseteq D$ such that $\one_\mathcal{D} \in a^*( \mathcal{D} )$ and for all $m \in D$ with $m \ne \one_\mathcal{D}$, $a^*( \mathcal{D} )$ contains an even number of elements greater than or equal to $m$.
\end{lemma}

\begin{proof}
Label the elements of $D$ as $d_1, \dots, d_n$ such that $d_i \ge d_j$ only if $i \le j$ (in particular, $d_1 = \one_\mathcal{D}$ and $d_n = \zero_\mathcal{D}$). We note that for all $i \ge 2$, $d_i^+ \subseteq \{ d_1, \dots, d_{i-1} \}$. We construct $a^*( \mathcal{D} )$ recursively as follows. Firstly, $d_1=\one_\mathcal{D} \in a^*( \mathcal{D} )$. Then, for all $2 \le i \le n$, $d_i \in a^*( \mathcal{D} )$ if and only if $|d_i^+ \cap a^*( \mathcal{D} )|$ is odd. We note that any other choice will violate either condition of the lemma.
\end{proof}

Now for any element $l$ of a bounded lattice $L$ let us write $a^l = a^*( l^\downarrow )$. In particular, $a^\one = a^*(\mathcal{L} )$. Note that (by construction) $l \in a^l$ and moreover, all elements of $a^l$ are less than or equal to $l$. Recall that we view subsets of $L$, and hence in particular $a^l$, as  elements of $\mathbb{F}_2L$.  Thus $\bigvee a^l = l$.

\begin{theorem} \label{theorem:atoms}
The set of atoms of $V( \mathcal{L} )$ is given by
\[
    \Atoms( V( \mathcal{L} ) ) = \{ a^l : l \in L \}.
\]
\end{theorem}
We shall see in the proof that  $b=a^l$ is the unique element of $V(\mathcal{L})$ satisfying $\bigvee b =l$.
We begin the proof by a simple but useful property of atoms of $V( \mathcal{L} )$.

\begin{lemma} \label{lemma:a_atom}
A nonzero element $a \in V( \mathcal{L} )$ is an atom if and only if $e_l a \in \{a, 0\}$ for all $l \in L$. 
\end{lemma}

\begin{proof}
Since $a$ is nonzero, it is an atom if and only if $x a \in \{ a, 0 \}$ for all $x \in V( \mathcal{L} )$. If $a$ is an atom, then in particular this holds for any unit vector $e_l$. Conversely, if $e_l a \in \{a, 0\}$ for all $l \in L$, then for any $x = \sum_l \alpha_l e_l$ we have $xa = \sum_l \alpha_l (e_la) \in \{ a, 0 \}$.
\end{proof}

\begin{corollary} \label{corollary:ela=0}
If $a$ is an atom of $V( \mathcal{L} )$, then $e_l a = 0$ for all $l < \bigjoin a$.
\end{corollary}

\begin{proof}
If $l < \bigjoin a$, then $l$ cannot be greater than or equal to all elements of $a$, that is there exists $m \in a \setminus l^\downarrow$. Recalling that all elements of $e_la$ lie in $l^\downarrow$ we then have $m \notin e_l a$ and hence $e_l a \ne a$; we obtain $e_l = 0$.
\end{proof}

\begin{proof}[Proof of Theorem \ref{theorem:atoms}]
    We first prove that if $a$ and $b$ are distinct atoms then $\bigjoin a \ne \bigjoin b$. Arguing for contradiction, suppose $a$ and $b$ are atoms with $a \ne b$ so that $a \Delta b \ne \emptyset$, but $s = \bigjoin a = \bigjoin b$. Let $l$ be a maximal element of $a \Delta b$, i.e. $l$ lies in exactly one of $a$ or $b$ and for any $m \in l^+$, either $m \in a \cap b$ or $m \notin a \cup b$. We obtain $|l^+ \cap a| = |l^+ \cap b|$ and hence $|l^\uparrow \cap a| = |l^\uparrow \cap b| \pm 1$. Thus either $e_l a \ne 0$ (if $|l^\uparrow \cap a|$ is odd) or $e_l b \ne 0$ (if $|l^\uparrow \cap a|$ is even). On the other hand, since $l < s$, we have $e_l a = e_l b = 0$ by Corollary \ref{corollary:ela=0}, which is the desired contradiction.

    We have just proved that the mapping $\Atoms( V( \mathcal{L} ) ) \to L$, $a \mapsto \bigjoin a$ is injective. Recalling that $V(\mathcal{L})$ is isomorphic to the power set lattice of $L$, we have $| \Atoms( V( \mathcal{L} ) ) | = | L | = n$, this mapping is also surjective, that is, for every $l \in L$, there is an atom of $V( \mathcal{L} )$, say $b$, such that $\bigjoin b = l$. Note that (by assumption) all elements of $b$ must be in $l^\downarrow$. We will now prove that $b = a^l$.

    First, we prove that $l \in b$. For the sake of contradiction, suppose $l \notin b$. Since $l = \bigjoin b$, there are at least two distinct maximal elements of $b$, say $x$ and $y$. We shall prove that $e_x b \notin \{ b, 0 \}$. Since $x, y \in L$ are maximal elements of $b$ we have $y \not \in x^{\downarrow}$, hence $y \notin e_x b$ and $e_x b \ne b$. Next,  $x^\uparrow \cap b = \{ x \}$, and so $x \in e_x b$ and $e_x b \ne 0$. Thus $e_x b \notin \{ b, 0 \}$, which contradicts Lemma \ref{lemma:a_atom}. From this contradiction, we then conclude that $l \in b$. Second, for all $m < l$ we have $l \notin e_m b$, hence $e_m b = 0$ by Lemma \ref{lemma:a_atom}. Thus, $|m^\uparrow \cap b|$ is even for all $m < l$. Finally, $b$ is a subset of $l^\downarrow$ such that $l \in b$ and $| m^\uparrow \cap b |$ is even for every $m < l$, thus by Lemma~\ref{lemma:a_star} we have $b = a^l$.  
\end{proof}

\begin{example} \label{example:hexagon}

An example of a lattice $\mathcal{L}$ is given on the left hand side Figure \ref{figure:hexagon}, with the atoms of $V( \mathcal{L} )$ listed on the right hand side. We remark that $a^\zero = e_\zero$, and that $a^t = t + \zero$ for each atom $t \in \Atoms( \mathcal{L} )$, two facts that hold for any finite lattice $\mathcal{L}$.

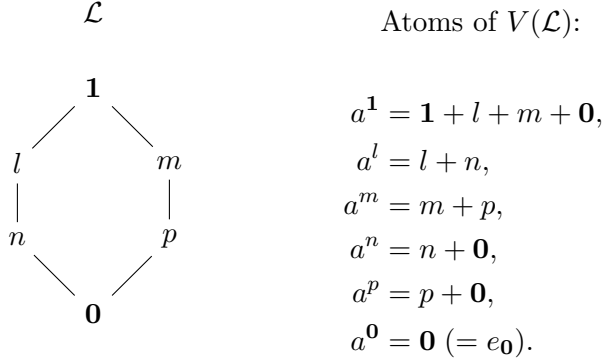
\begin{figure}[!hb]
    \centering

    \begin{tikzpicture}
        \node (L) at (1,4) {$\mathcal{L}$};
    
        \node (1) at (1,3) {$\one$};
        \node (l) at (0,2) {$l$};
        \node (m) at (2,2) {$m$};
        \node (n) at (0,1) {$n$};
        \node (p) at (2,1) {$p$};
        \node (0) at (1,0) {$\zero$};
    
        \draw (0) -- (n);
        \draw (0) -- (p);
        \draw (n) -- (l);
        \draw (p) -- (m);
        \draw (l) -- (1);
        \draw (m) -- (1);
    
        \node[text width = 10cm] at (6,2) {\begin{align*} & \text{Atoms of $V( \mathcal{L} )$:} \\~\\ a^\one &= \one + l + m + \zero, \\ a^l &= l + n, \\ a^m &= m + p, \\ a^n &= n + \zero, \\ a^p &= p + \zero, \\ a^\zero &= \zero \;(= e_\zero).
        \end{align*} };
        
    \end{tikzpicture}
    
    \caption{The hexagon lattice $\mathcal{L}$ and the atoms of $V( \mathcal{L} )$.}
    \label{figure:hexagon}
\end{figure}
\end{example}

We now generalise to the case where $G$ is a semilattice, and not only a lattice. Actually, the framework to study this case also encompasses the sublattice of sums containing an even number of terms, which will be useful when we study the division of injective partial transformations in Section \ref{subsection:division_LC}. As such, we give the framework in its generality first, and then focus on those two important special cases.

We now have a closer look at a particular kind of sublattice of $V( \mathcal{L} )$. For any $x \in V( \mathcal{L} )$, and any $l \in L$, let $|x|_l = 1$ if $x \ge a^l$ and $|x|_l = 0$ otherwise. Since $a^l$ is an atom, $x a^l \in \{ a^l, 0 \}$; we obtain that $|x|_l =1$ if $xa^l = a^l$ and $|x|_l = 0$ if $x a^l = 0$. We then define
\[
    P_l( \mathcal{L} ) = \{ x \in V( \mathcal{L} ) : |x|_l = 0 \}.
\]

\begin{lemma} \label{lemma:Pl}
For any $l \in L$, we have $P_l( \mathcal{L} ) = ( a^l + e_\one )^\downarrow$.
\end{lemma}

\begin{proof}
We have $|x|_l = 0$ if and only if $x a^l = 0$, which in turn is equivalent to $x \le \neg a^l = a^l + e_\one$.
\end{proof}

This immediately yields a simple characterisation of the intersection between an interval and $P_l( \mathcal{L} )$, which we shall use in Section \ref{section:injective}.

\begin{corollary} \label{corollary:Pl}
For all $\lambda, \upsilon \in V(\mathcal{L})$ with $\lambda \le \upsilon$ and all $l \in L$, we have
\begin{align}
    \label{equation:P1a}
    [ \lambda, \upsilon ] \cap P_l( \mathcal{V} ) &= [ \lambda, \upsilon + \upsilon a^l ], \\
    \label{equation:P1b}
    [ \lambda, \upsilon ] \setminus P_l( \mathcal{V} ) &= [ \lambda + a^l + \lambda a^l, \upsilon ].
\end{align}
\end{corollary}

\begin{proof}
We have
\[
    [ \lambda, \upsilon ] \cap P_l( \mathcal{V} ) = [ \lambda, \upsilon ] \cap [ 0, a^l + e_\one ] 
    = [ \lambda \join 0, \upsilon ( a^l + e_\one ) ] 
    = [ \lambda, \upsilon + \upsilon a^l ].
\]
And similarly, $[ \lambda, \upsilon ] \setminus P_l( \mathcal{V} ) = [ \lambda, \upsilon ] \cap [ a^l, e_\one ] = [ \lambda \join a^l, \upsilon e_\one ] = [ \lambda + a^l + \lambda a^l, \upsilon]$.

\end{proof}

We focus on two particular cases of $P_l( \mathcal{L} )$. First, the \Define{parity-check code} is defined as the elements that are not greater than $a^\zero = e_\zero$:
\[
    P_\zero( \mathcal{L} ) = \{ x \in V( \mathcal{L} ) : |x|_\zero = 0 \}.
\]
According to Lemma \ref{lemma:Pl}, $P_\zero( \mathcal{L} ) = (e_\one + e_\zero)^\downarrow$. The parity-check code is then a Boolean algebra with atoms $\Atoms( V( \mathcal{L} ) ) \setminus \{ e_\zero \}$. The term parity-check code comes from the fact that $|a|_\zero$ is the cardinality of $a$, viewed a subset of $L$, computed modulo $2$: the parity-check code is then the family of all the subsets of even size.

Second, removing the upper bound $\one$ from $\mathcal{L}$ yields a meet-semilattice, and conversely, for any finite meet-semilattice $\mathcal{S}$, $\mathcal{S}^\one$ obtained by adding an upper bound $\one$ is a lattice. We can then generalise the $V( \mathcal{L} )$ construction to finite meet-semilattices. Indeed, let $\mathcal{L} = \mathcal{S}^\one$ and
\[
    V( \mathcal{S} ) = P_\one( \mathcal{S}^\one ) = \{ x \in V( \mathcal{S}^\one ) : |x|_\one = 0 \}.
\]

\begin{theorem} \label{theorem:VS}
Let $\mathcal{S}$ be a finite meet-semilattice. Let $a^\one$ be the unique atom of $V( \mathcal{S}^\one )$ such that $\bigjoin a^\one = \one$. Then $V( \mathcal{S} ) = ( \F_2\mathcal{S}, \join, \meet, \one_{ V( \mathcal{S}^\one ) }, \zero_{ V( \mathcal{S}^\one ) } )$, where $\one_{ V( \mathcal{S}^\one ) } = e_\one + a^\one$ and $\zero_{ V( \mathcal{S}^\one ) } = 0$ is the empty sum, is the Boolean algebra with set of atoms $\Atoms( V( \mathcal{S}^\one ) ) \setminus \{ a^\one \}$.
\end{theorem}

\begin{proof}
We only need to prove that the set of elements of $V( \mathcal{S} )$ is $\F_2 \mathcal{S}$. Let $x \in V( \mathcal{S}^\one )$. If $\one \notin x$, then $x a^\one = 0$ by Corollary \ref{corollary:ela=0}; conversely, if $\one \in x$, then $\one \in x a^\one$ and hence $x a^\one = a^\one$. We obtain that $|x|_\one = 0$ if and only if $\one \notin x$, i.e. $x$ is a formal sum of terms in $\mathcal{S}$.
\end{proof}

As an example, and anticipating the next section, $V( \mathcal{S}^\one )$ for the semilattice $\mathcal{S} = \{ C_3, C_5, C_{15} \}$ is given in Figure \ref{figure:boolean_semilattice}, where the parity-check code $P_\zero( \mathcal{L} )$ and the Boolean algebra $V( \mathcal{S} )$ are also displayed.

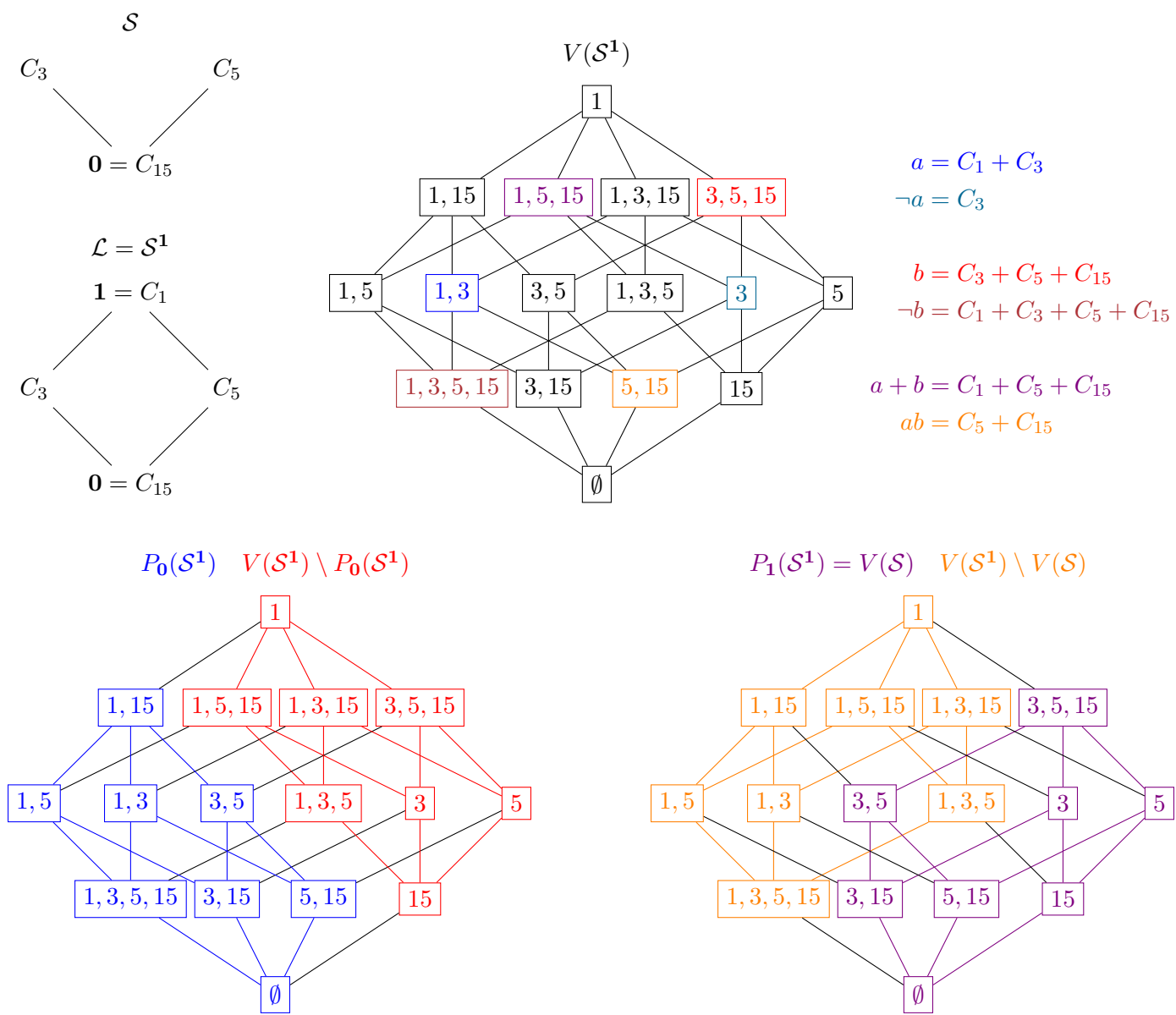
\begin{figure}[bp]
\centering
\begin{tikzpicture}

    \begin{scope}[xshift=-3cm, yshift=0cm, xscale = 1.5, yscale = 1.5]
    \node (C15) at (1,0) {$\zero = C_{15}$};
    \node (C3) at (0,1) {$C_3$};
    \node (C5) at (2,1) {$C_5$};

    \path[draw]
    (C15) edge (C3)
    (C15) edge (C5);

    \node (S) at (1,1.5) {$\mathcal{S}$};
    \end{scope}

    \begin{scope}[xshift=-3cm, yshift=-5cm, xscale = 1.5, yscale = 1.5]
    \node (C15) at (1,0) {$\zero = C_{15}$};
    \node (C3) at (0,1) {$C_3$};
    \node (C5) at (2,1) {$C_5$};
    \node (C1) at (1,2) {$\one = C_1$};

    \path[draw]
    (C15) edge (C3)
    (C15) edge (C5)
    (C5) edge (C1)
    (C3) edge (C1);

    \node (L) at (1,2.5) {$\mathcal{L} = \mathcal{S}^\one$};
    \end{scope}

\begin{scope}[xshift= 5cm, yshift =-5cm, xscale = 1.5, yscale = 1.5]
    \node (V) at (0.5,4.5) {$V( \mathcal{S}^\one )$};

    \node[draw] (0000) at (0.5,0) {$\emptyset$}; 
    
    \node[draw, Maroon] (0001) at (-1,1) {$1, 3, 5, 15$}; 
    \node[draw] (0010) at (0,1) {$3,15$}; 
    \node[draw, orange] (0100) at (1,1) {$5,15$}; 
    \node[draw] (1000) at (2,1) {$15$}; 
    
    \node[draw] (0011) at (-2,2) {$1,5$}; 
    \node[draw, blue] (0101) at (-1,2) {$1,3$}; 
    \node[draw] (1001) at (1,2) {$1,3,5$}; 
    \node[draw] (0110) at (0,2) {$3,5$}; 
    \node[draw, MidnightBlue] (1010) at (2,2) {$3$}; 
    \node[draw] (1100) at (3,2) {$5$}; 
    
    \node[draw] (0111) at (-1,3) {$1,15$}; 
    \node[draw, violet] (1011) at (0,3) {$1,5,15$}; 
    \node[draw] (1101) at (1,3) {$1,3,15$}; 
    \node[draw, red] (1110) at (2,3) {$3,5,15$}; 
    
    \node[draw] (1111) at (0.5,4) {$1$}; 

    \draw (0000) -- (0001);
    \draw (0000) -- (0010);
    \draw (0000) -- (0100);
    \draw (0000) -- (1000);
    
    \draw (0001) -- (0011);
    \draw (0001) -- (0101);
    \draw (0001) -- (1001);
    
    \draw (0010) -- (0011);
    \draw (0010) -- (0110);
    \draw (0010) -- (1010);
    
    \draw (0100) -- (0101);
    \draw (0100) -- (0110);
    \draw (0100) -- (1100);
    
    \draw (1000) -- (1001);
    \draw (1000) -- (1010);
    \draw (1000) -- (1100);

    \draw (0011) -- (0111);
    \draw (0011) -- (1011);

    \draw (0101) -- (0111);
    \draw (0101) -- (1101);

    \draw (1001) -- (1011);
    \draw (1001) -- (1101);

    \draw (0110) -- (0111);
    \draw (0110) -- (1110);

    \draw (1010) -- (1011);
    \draw (1010) -- (1110);

    \draw (1100) -- (1101);
    \draw (1100) -- (1110);

    \draw (1111) -- (1110);
    \draw (1111) -- (1101);
    \draw (1111) -- (1011);
    \draw (1111) -- (0111);

    \node[text width = 2cm] (A) at (4,2) {$\begin{aligned} \textcolor{blue}{a\;} &\textcolor{blue}{= C_1 + C_3}\\ \textcolor{MidnightBlue}{\neg a\;} &\textcolor{MidnightBlue}{= C_3} \\ & \\ \textcolor{red}{b\;} &\textcolor{red}{= C_3 + C_5 + C_{15}} \\ \textcolor{Maroon}{\neg b\;} &\textcolor{Maroon}{= C_1 + C_3 + C_5 + C_{15}} \\ & \\ \textcolor{violet}{a+b\;} &\textcolor{violet}{= C_1 + C_5 + C_{15}}\\ \textcolor{orange}{ab\;} &\textcolor{orange}{= C_5 + C_{15}} \end{aligned}$};

\end{scope}

\begin{scope}[yshift=-13cm, xscale = 1.5, yscale = 1.5]

    \node[draw, blue] (0000) at (0.5,0) {$\emptyset$}; 
    
    \node[draw, blue] (0001) at (-1,1) {$1, 3, 5, 15$}; 
    \node[draw, blue] (0010) at (0,1) {$3,15$}; 
    \node[draw, blue] (0100) at (1,1) {$5,15$}; 
    \node[draw, red] (1000) at (2,1) {$15$}; 
    
    \node[draw, blue] (0011) at (-2,2) {$1,5$}; 
    \node[draw, blue] (0101) at (-1,2) {$1,3$}; 
    \node[draw, red] (1001) at (1,2) {$1,3,5$}; 
    \node[draw, blue] (0110) at (0,2) {$3,5$}; 
    \node[draw, red] (1010) at (2,2) {$3$}; 
    \node[draw, red] (1100) at (3,2) {$5$}; 
    
    \node[draw, blue] (0111) at (-1,3) {$1,15$}; 
    \node[draw, red] (1011) at (0,3) {$1,5,15$}; 
    \node[draw, red] (1101) at (1,3) {$1,3,15$}; 
    \node[draw, red] (1110) at (2,3) {$3,5,15$}; 
    
    \node[draw, red] (1111) at (0.5,4) {$1$}; 

    \draw[blue] (0000) -- (0001);
    \draw[blue] (0000) -- (0010);
    \draw[blue] (0000) -- (0100);
    \draw (0000) -- (1000);
    
    \draw[blue] (0001) -- (0011);
    \draw[blue] (0001) -- (0101);
    \draw (0001) -- (1001);
    
    \draw[blue] (0010) -- (0011);
    \draw[blue] (0010) -- (0110);
    \draw (0010) -- (1010);
    
    \draw[blue] (0100) -- (0101);
    \draw[blue] (0100) -- (0110);
    \draw (0100) -- (1100);
    
    \draw[red] (1000) -- (1001);
    \draw[red] (1000) -- (1010);
    \draw[red] (1000) -- (1100);

    \draw[blue] (0011) -- (0111);
    \draw (0011) -- (1011);

    \draw[blue] (0101) -- (0111);
    \draw (0101) -- (1101);

    \draw[red] (1001) -- (1011);
    \draw[red] (1001) -- (1101);

    \draw[blue] (0110) -- (0111);
    \draw (0110) -- (1110);

    \draw[red] (1010) -- (1011);
    \draw[red] (1010) -- (1110);

    \draw[red] (1100) -- (1101);
    \draw[red] (1100) -- (1110);

    \draw[red] (1111) -- (1110);
    \draw[red] (1111) -- (1101);
    \draw[red] (1111) -- (1011);
    \draw (1111) -- (0111);

    \node (A) at (0.5,4.5) {$\textcolor{blue}{ P_\zero( \mathcal{S}^\one ) } \quad \textcolor{red}{ V( \mathcal{S}^\one ) \setminus P_\zero( \mathcal{S}^\one ) } $};

\end{scope}
\begin{scope}[yshift=-13cm, xshift= 10cm, xscale = 1.5, yscale = 1.5]

    \node[draw, violet] (0000) at (0.5,0) {$\emptyset$}; 
    
    \node[draw, orange] (0001) at (-1,1) {$1, 3, 5, 15$}; 
    \node[draw, violet] (0010) at (0,1) {$3,15$}; 
    \node[draw, violet] (0100) at (1,1) {$5,15$}; 
    \node[draw, violet] (1000) at (2,1) {$15$}; 
    
    \node[draw, orange] (0011) at (-2,2) {$1,5$}; 
    \node[draw, orange] (0101) at (-1,2) {$1,3$}; 
    \node[draw, orange] (1001) at (1,2) {$1,3,5$}; 
    \node[draw, violet] (0110) at (0,2) {$3,5$}; 
    \node[draw, violet] (1010) at (2,2) {$3$}; 
    \node[draw, violet] (1100) at (3,2) {$5$}; 
    
    \node[draw, orange] (0111) at (-1,3) {$1,15$}; 
    \node[draw, orange] (1011) at (0,3) {$1,5,15$}; 
    \node[draw, orange] (1101) at (1,3) {$1,3,15$}; 
    \node[draw, violet] (1110) at (2,3) {$3,5,15$}; 
    
    \node[draw, orange] (1111) at (0.5,4) {$1$}; 

    \draw (0000) -- (0001);
    \draw[violet] (0000) -- (0010);
    \draw[violet] (0000) -- (0100);
    \draw[violet] (0000) -- (1000);
    
    \draw[orange] (0001) -- (0011);
    \draw[orange] (0001) -- (0101);
    \draw[orange] (0001) -- (1001);
    
    \draw (0010) -- (0011);
    \draw[violet] (0010) -- (0110);
    \draw[violet] (0010) -- (1010);
    
    \draw (0100) -- (0101);
    \draw[violet] (0100) -- (0110);
    \draw[violet] (0100) -- (1100);
    
    \draw (1000) -- (1001);
    \draw[violet] (1000) -- (1010);
    \draw[violet] (1000) -- (1100);

    \draw[orange] (0011) -- (0111);
    \draw[orange] (0011) -- (1011);

    \draw[orange] (0101) -- (0111);
    \draw[orange] (0101) -- (1101);

    \draw[orange] (1001) -- (1011);
    \draw[orange] (1001) -- (1101);

    \draw (0110) -- (0111);
    \draw[violet] (0110) -- (1110);

    \draw (1010) -- (1011);
    \draw[violet] (1010) -- (1110);

    \draw (1100) -- (1101);
    \draw[violet] (1100) -- (1110);

    \draw (1111) -- (1110);
    \draw[orange] (1111) -- (1101);
    \draw[orange] (1111) -- (1011);
    \draw[orange] (1111) -- (0111);

    \node (A) at (0.5,4.5) {$\textcolor{violet}{ P_\one( \mathcal{S}^\one ) = V( \mathcal{S} ) } \quad \textcolor{orange}{ V( \mathcal{S}^\one ) \setminus V( \mathcal{S} ) } $};
\end{scope}
\end{tikzpicture}
\caption{Boolean algebra $V( \mathcal{S}^\one )$, with the corresponding parity-check code $P_\zero( \mathcal{S}^\one )$ and Boolean algebra $V( \mathcal{S} )$.}
\label{figure:boolean_semilattice}
\end{figure}

\section{The semiring generated by $C$ with $\F_2$} \label{section:semiring_C2}

In this section, we consider ``permutations computed modulo $2$.'' A permutation is a sum of cycles, where $C_d = d \to d-1 \to \dots \to 1 \to d$ is the \Define{cycle} of length $d$ for all $d \ge 1$. Therefore, let $C$ be the set of cycles and consider $S = \F_2C$.

\subsection{Basic properties} \label{subsection:S_basic_properties}

Clearly $S$ is a ring, with multiplicative identity $C_1$. In $S$ we have \cite[Proposition 3]{DFMMR19}
$$C_mC_n = {\rm gcd}(m,n)C_{\lcm(m,n)} = \begin{cases} C_{\lcm(m,n)} & \mbox{ if at least one of } m \mbox{ and }n \mbox{ is odd,}\\
0& \mbox{else.}\end{cases}$$
We begin by gathering some basic facts concerning the structure of $S$. Notation: For a positive integer $n$ we shall write $d(n) = {\rm max}\{m: 2^m \mid n\}$. If $d(n)=m$, then $C_n = C_{2^m}C_{n'}$ where $n=2^mn'$ and $n'$ is an odd positive integer. 

Let
\[
    S_0 = \left\{ \sum_{ k \text{ odd} } \alpha_k C_k \in S : \alpha_k \in \{0,1\}, \text{ finite number of nonzero } \alpha_k \right\},
\]
be the set of elements in $S$ where all the cycle lengths are odd.

Then, any element $s \in S$ can be uniquely written as 
\[
    s = \sum_{i \in \N} s_i C_{ 2^i },
\]
where $s_i \in S_0$ for all $i$ (and there is a finite number of nonzero $s_i$'s). We also denote the elements of $S$ that only have even cycles as 
\[
    S_+ = \left \{ \sum_{i \in \N} s_i C_{ 2^i} \in S: s_0 = 0 \right \}.
\]
We remark that $S_+$ is the multiplicative ideal generated by $C_2, C_4, C_8, \dots$.
Then, any element $s \in S$ can be uniquely written as $s = s_0 + s_+$, with $s_0 \in S_0$ and $s_+ \in S_+$. It is easily seen that for all $s,t \in S$
\[
     s_+t_+ = 0, \quad s^2 = s_0^2=s_0, \quad s_+ = s + s^2.
\]
The next lemma records several facts resulting from these simple  observations.

\begin{lemma}
\label{lem:basis}
Let $x,y \in S$.
\begin{enumerate}
\item \label{item:idempotent} $x$ is idempotent if and only if $x$ is a sum of odd length cycles.
\item \label{item:decomposition} If $x \neq 0$, then there exists $k \geq 1$ such that $x$ can be uniquely written as $x=\sum_{i=0}^k x_i C_{2^i}$ where each $x_i$ is idempotent and $x_k \neq 0$.
\item \label{item:x4=x2} For $x=\sum_{i=0}^k x_i C_{2^i}$, we have $x^2 = x_0$. Consequently, for any $x \in S$ we have $x^4 = x^2$.

\item \label{item:xy} For $x=\sum_i x_i C_{2^i}$ and $y = \sum_i y_i C_{2^i}$, we have $xy = z = \sum_i z_i C_{2^i}$ with $z_0 = x_0  y_0$ and $z_i = x_0 y_i + x_i y_0$ for all $i \ge 1$.

\item \label{item:unit} $x$ is a unit if and only if $x_0 = C_1$ (and in this case $x^2=C_1$).

\item \label{item:complement} $x_0 \leq y_0$ if and only if $x_0 + 1 \geq y_0 + 1$.
\end{enumerate}
\end{lemma}

\begin{proof}
\begin{enumerate}
    \item \label{item:proof_idempotent}
First suppose that $x$ is a sum of odd length cycles only, say $x = C_{i_1} + \cdots + C_{i_k}$ where each $i_j$ is odd. Then, since we are in characteristic $2$ we have: 
$$x^2 = (C_{i_1} + \cdots + C_{i_k})^2 = C_{i_1}^2 + \cdots + C_{i_k}^2 = C_{i_1} + \cdots + C_{i_k}.$$
Next, for an arbitrary element $x \in S$ we may write $x = x_0 + x_+$ where $x_0$ is a sum of odd cycles only and $x_+$ is a sum of even cycles only. Since the product of any two even length cycles is 0 it follows that $(x_+)^2=0$ and hence that $x^2 = (x_0)^2 + (x_+)^2 = x_0 + 0 = x_0$. Thus $x$ is idempotent if and only if $x$ is a sum of odd length cycles only.

    \item \label{item:proof_decomposition}
By definition of $S$ we have that $x$ may be uniquely written as $x = C_{i_1} + C_{i_2} + \cdots + C_{i_t}$  for some $1 \leq i_1 < i_2 < \cdots <i_t$. Let $M = \{d(i_1), \ldots, d(i_t)\}$ where for each $1 \leq j \leq t$ we recall that $d(i_j)$ denotes the maximal power of $2$ dividing $i_j$. For $m \in M$ let $I_m = \{s: d(i_s)=m\}$ and note that for each $s \in I_m$ write $i_s=2^mi'_s$ where $i'_s$ is odd. Then we may reorder the terms in this unique expression as follows:  
$$x = \sum_{m \in M} \sum_{s \in I_m} C_{i_s} = \sum_{m \in M} \sum_{s \in I_m} C_{2^m}C_{i'_s}= \sum_{m \in M} C_{2^m}\sum_{s \in I_m} C_{i'_s}.$$
Setting $x_m = \sum_{s \in I_m} C_{i'_s}$ it is then clear that each $x_m$ is a sum of cycles of odd length only, and hence by part (\ref{item:idempotent}), an idempotent. One can add in terms with $0$ coefficients to obtain the required form.

    \item \label{item:proof_x4=x2}
As noted in the proof of part \ref{item:proof_idempotent}, if $x = x_0 + x_+$ where $x_0$ is a sum of cycles of odd length only and $x_+$ is a sum of cycles of even length only, then $x^2 = x_0$. Thus $x^4 = (x^2)^2 = (x_0)^2 = x_0 = x^2$.

    \item \label{item:proof_xy}
We have $z = ( x_0 + \sum_{i \ge 1} x_i C_{2^i} ) ( y_0 + \sum_{j \ge 1} y_j C_{2^j} ) $. Since $C_{2^i} C_{2^j} = 0$ if $i,j > 0$, we obtain $z = x_0 y_0 + \sum_{i \ge 1} x_i y_0 C_{2^i} + \sum_{j \ge 1} x_0 y_j C_{2^j}$. 

    \item \label{item:proof_unit}
Suppose that $x$ is a unit. Then $xy = C_1$ for some $y$, and by part (\ref{item:xy}), $x_0y_0 = C_1$. Thus, $x_0 = x_0 x_0 y_0 = x_0 y_0 = C_1$. Conversely, if $x_0 = C_1$, then $x^2 = x_0 = C_1$, demonstrating that $x$ is unit.

\item For all $x_0, y_0 \in S_0$ we have $x_0 + 1 \leq y_0 +1$ if and only if $(x_0+1)(y_0+1) = x_0y_0 + x_0 + y_0 +1 = x_0 + 1$ if and only if $x_0y_0 + y_0 = 0$ if and only if $x_0y_0 = y_0$ if and only if $x_0 \geq y_0$.
\end{enumerate}

\end{proof}

\subsection{Boolean algebra of idempotents} \label{subsection:boolean_algebra_S0}

We first note that the set $S_0$ of idempotent elements forms an infinite Boolean algebra.

\begin{lemma}
\label{lem:lattice}
The set $E(S) = S_0$ of idempotent elements of $S$ forms an atomless Boolean algebra, where for $e, f \in E(S)$  we have $e \meet f = ef$, $e \join f = e + f + ef$ and $\neg e = C_1+e$.
\end{lemma}

\begin{proof}
Since $S$ is commutative, it is clear that a product of idempotents is idempotent. Moreover, since $S$ is of characteristic $2$, it is also easy to see that a sum of idempotents is idempotent. Thus $E(S)$ is a subring of $S$, and clearly $E(S)$ is a Boolean ring. It well known \cite[Chapter 2]{GH09} that any Boolean ring can be viewed as a Boolean algebra with respect to the given operations.  

Let us show that it is atomless. Let $s \in E(S)$ be a nonzero element, say $s = \sum_i \alpha_i C_i$ and let $p$ be a prime number that does not divide any cycle length in $s$. Then $t = sC_p = \sum_i \alpha_i C_{ip}$ satisfies $t \notin \{ 0, s \}$ and $t \le s$ (i.e. $ts = t$), hence $s$ is not an atom.
\end{proof}

\begin{remark}
\label{rem:atomlessB}
There is a unique (up to isomorphism) countably infinite atomless Boolean algebra, e.g. the interval algebra of the rationals \cite[Chapter 16]{GH09}. Therefore, $(E(S), \meet, \join, \neg)$ is isomorphic to that interval algebra.
\end{remark}

\paragraph{Dividing in $S_0$.}

By using the Boolean algebra structure of $S_0$, we can easily solve the equation $ax = b$ when $a, b \in S_0$ only have odd cycle lengths. Note that if $ax = b$ with $a,b \in S_0$, then necessarily $x \in S_0$ as well (unless $a = b = 0$, in which case any $x \in S$ is a solution).

We shall need the general division formula for any Boolean algebra $\mathcal{V}$ (see Section \ref{section:injective}, notably Theorem \ref{theorem:division_Lh} and Theorem \ref{theorem:division_LC}). We remark that in any Boolean algebra, $a \ge b$ is equivalent to $b \le a + b + \one_\mathcal{V}$; as such, the interval $[b, a + b + \one_\mathcal{V} ]$ is non-empty if and only if $ab = b$.

\begin{proposition} \label{proposition:ax=b_S0}
Let $\mathcal{V}$ be a Boolean algebra and $a,b \in \mathcal{V}$. Then the solutions to $a x = b$ are the interval $[ b, a + b + \one_\mathcal{V} ]$ in $\mathcal{V}$.
\end{proposition}

\begin{proof}
There exists $x$ such that $ax = b$ if and only if $ab = b$ (since $ax = b$ implies $b = ax = aax = ab$). Assume then that $ab = b$. If $ax = b$, then by a similar argument as above, we have $xb = b$. Therefore, $x \ge b$ and $x ( b + a + \one_\mathcal{V} ) = b + b  +x= x$, hence $x \le a + b + \one_\mathcal{V}$. Conversely, if $b \le x \le a + b + \one_\mathcal{V}$, then $( a + b + \one_\mathcal{V} ) x = x$ and $bx = b$, giving
\[
    ax = ax + bx + x + bx + x = ( a + b + \one_\mathcal{V} )x + bx + x = b.
\]
\end{proof}

\paragraph{Finite case.}

We now focus on the finite lattice generated by odd cycles that divide a given length. Let $k$ be an odd positive integer, and let $D(k)$ denote the set of divisors of $k$. Hence we consider the lattice $\mathcal{C}_k = \{ C_i : i \ | \ k  \}$, with the meet operation $C_i C_j = C_{\lcm(i,j)}$. For instance, the lattice $\mathcal{C}_{15}$ is displayed in Figure \ref{figure:boolean_semilattice}. The elements of $S$ whose cycle lengths divide $k$ then form the lattice $V( \mathcal{C}_k )$. Theorem \ref{theorem:atoms} then gives an abstract expression for its atoms; however, we can be much more explicit, as seen below.

The lattice $\mathcal{C}_k$ is isomorphic to the lattice $\mathcal{L} = ( L = D(k), \join = \mathrm{gcd}, \meet = \lcm, \zero = k, \one = 1 )$. Recall that for element $j \in L$, $j^c$ denotes the set of elements of $L$ covered by $j$. In the lattice $\mathcal{L}$, for any $j \in D(k)$ we denote $m(j) = \bigmeet j^c$. For instance, for $k = 45$, $j = 3$ covers $9$ and $15$, so that $j^c = \{9, 15\}$ and $m(j) = \lcm(9, 15) = 45$ (see Example \ref{example:C45} below). We note that the interval $[m(j), j] = \{ l : j \ | \ l \text{ and } l \ | \ m(j) \}$ is never empty. Now, for any $j \in D(k)$, let $T_j \in \mathcal{C}_k$ be defined as 
\[
    T_j = \sum_{ l \in [m(j), j] } C_l.
\]

\begin{theorem} \label{theorem:Ti_atoms}
The set of atoms of $V( \mathcal{C}_k )$ is $\{ T_j : j \in D(k) \}$. In particular, for all $i \in D(k)$,
\[
    C_i = \sum_{j \in D(k) : i \ | \ j} T_j.
\]

\end{theorem}

\begin{proof} Since the $T_j$'s are nonzero, Lemma \ref{lemma:a_atom} shows that we only need to prove that $C_i T_j \in \{T_j, 0\}$ for all $i,j \in D(k)$.
We have $T_j = \sum_{l \in [m(j), j]} C_l$. Firstly, if $i$ divides $j$, then $i$ divides any $l$ in $[m(j), j]$, hence $C_i C_l = C_l$ and $C_i T_j = T_j$. Henceforth, we assume that $i$ does not divide $j$.

Suppose $k$ is factored as $k = \prod_{p \in P} p^{k_p}$ where $P$ is the set of prime factors of $k$; for notational convenience we shall view it as the vector $k = (k_p : p \in P)$ and more generally we shall write an element $t \in D(k)$ as $t = (t_p : p \in P)$ where $0 \le t_p \le k_p$.

Suppose then that $j = (j_p : p \in P)$ and $i = (j_p : p \in P)$ with $0 \le j_p \le k_p$ and $0 \le i_p \le k_p$. Let $\Sigma = \{ p \in P : i_p > j_p \}$. Since $i$ does not divide $j$ we have $|\Sigma|>0$. Let $\Theta = \{ p : j_p = k_p \}$ and $\Gamma = P \setminus \Theta$. Note also that  $\Sigma \subseteq \Gamma$. Writing $m(j) = (m_p : p \in P)$, for each $p \in P$ we have
\[
    m_p = \min\{ j_p + 1, k_p \} = \begin{cases}
        j_p &\text{if } p \in \Theta \\
        j_p + 1 &\text{if } p \in \Gamma.
    \end{cases}
\]
Let $l \in [m(j),j]$ so that $l_p = j_p + l(\epsilon)_p$ for all $p \in \Gamma$ and $l(\epsilon)_p \in \{0,1\}$. Then writing $q = \lcm(l,i) = (q_p : p \in P)$ we have
\[
    q_p = \max\{ i_p, l_p \} = \begin{cases}
        i_p &\text{if } p \in \Sigma \\
        l_p = j_p + l(\epsilon)_p &\text{if } p \in \Gamma \setminus \Sigma \\
        j_p &\text{if } p \in \Theta.
    \end{cases}
\]
Now let $Q$ denote the subset of $D(k)$ consisting of elements $(q_p: p \in P)$ with $q_p = i_p$ for all $p \in \Sigma$ and $q_p = j_p$ for all $p \in \Theta$.  
Then, for a given $q \in Q$ we see that $$A_q:=\{l \in [m(j), j] : \lcm(l,i) = q\} = \{ l \in [m(j),j] : l(\epsilon)_{\Gamma \setminus \Sigma}  = (q - j)_{\Gamma \setminus \Sigma} \}$$ and hence $|A| = 2^{| \Sigma | } \geq 2$. Putting this together, we obtain
\[
    C_i T_j = \sum_{l \in [m(j), j]} C_iC_l = \sum_{l \in [m(j), j] }C_{{ \rm lcm}(l,i)}  =\sum_{q \in Q} \sum_{l \in A_q} C_q = 0.
\]
Thus, 
\[
    C_i = \sum_{j \in D(k)} C_i T_j =  \sum_{j \in D(k) : i \ | \ j} T_j.
\]
\end{proof}

\begin{corollary}
For all $i$, $C_i = \sum_{j : i \ | \ j} T_j$.
\end{corollary}

\begin{example} \label{example:C45}
For $k = 45$, the lattice $\mathcal{C}_{45}$ and the atoms of $V( \mathcal{C}_{45} )$ are given below.

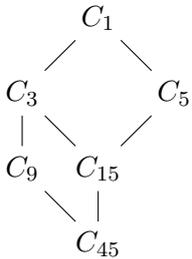
\begin{tikzpicture}
    \node (1) at (1,3) {$C_1$};
    \node (3) at (0,2) {$C_3$};
    \node (5) at (2,2) {$C_5$};
    \node (9) at (0,1) {$C_9$};
    \node (15) at (1,1) {$C_{15}$};
    \node (45) at (1,0) {$C_{45}$};

    \path[draw]
    (1) to (3)
    (1) to (5)
    (3) to (9)
    (3) to (15)
    (5) to (15)
    (9) to (45)
    (15) to (45);

    \node[text width = 10cm] (T) at (6,1.5) {\begin{align*}
    T_1 &= C_1 + C_3 + C_5 + C_{15} \\
    T_3 &= C_3 + C_9 + C_{15} + C_{45} \\
    T_5 &= C_5 + C_{15} \\
    T_9 &= C_9 + C_{45} \\
    T_{15} &= C_{15} + C_{45} \\
    T_{45} &= C_{45}.
\end{align*}};

    \node[text width = 10 cm] (C) at (12,1.5) {
    \begin{align*}
    C_1 &= T_1 + T_3 + T_5 + T_9 + T_{15} + T_{45} \\
    C_3 &= T_3 + T_9 + T_{15} + T_{45} \\
    C_5 &= T_5 + T_{15} + T_{45} \\
    C_9 &= T_9 + T_{45} \\
    C_{15} &= T_{15} + T_{45} \\
    C_{45} &= T_{45}.
\end{align*}
    };
\end{tikzpicture}

\end{example}

\subsection{Division in $S$} \label{subsection:division_in_S}

We now consider the division equation $ax = b$ in $S$.
First of all, it is worth noting that in general, there could be an infinite number of solutions, e.g. $C_2 x = 0$ holds for any $x \in S_+$. This even holds for odd lengths only: $(C_3 + C_5)x = 0$ holds for instance for any $x \in S C_{15}$. But we could ask for solutions of bounded size. In fact, we prove below that if a solution exists, then it always contains a solution  without any ``artefacts'', i.e. cycle lengths that one would not naturally expect from a solution. 

For all $a \in S$, let $Q_a$ denote the set of cycle lengths that appear in $a$, so that $a = \sum_{q \in Q_a} C_q$. Recall that $q = q' 2^{ d( q ) }$ for any $q$; then let $K(a) = \lcm( q' : q \in Q_a )$ and $N(a) = \max( d(q) : q \in Q_a )$. For instance, if $a = (C_1 + C_5) + C_7 C_2 + C_5C_8$, we have $K(a) = \lcm( 1, 5, 7, 5 ) = 35$ and $N(a) = \log_2 8 = 3$.

Now, for any $a = \sum_{q \in Q_a} C_q$ and $n, k \in \N$, let $a_{-n}$ denote the element obtained from $a$ by omitting any cycles $C_q$ with lengths divisible by $2^{q+1}$ and let $a_{|k}$ denote the element obtained from $a$ by omitting any cycles $C_q$ with lengths not divisible by $k$. Thus
\begin{align*}
    a_{-n} &= \sum_{q \in Q_a : d(q) \le n} C_q =  \sum_{i \le n} a_i C_{2^i},\\
    a_{ | \ k} &= \sum_{q \in Q_a : q \ | \ k} C_q.
\end{align*}
We note that $(a_{ -n})_{| \ k} = (a_{| \ k})_{-n}$.

\begin{proposition} \label{proposition:division_small_cycles}
Let $a,b, x \in S$, and $n, k \in \N$ be such that $\max( N(a), N(b) ) \le n$ and $\lcm( K( a ) , K( b ) ) \ | \ k$. Then $ax = b$ only if $a (x_{-n})_{| \ k} = b$.
\end{proposition}

\begin{proof}
We first note that, by assumption, $a = (a_{-n})_{| \ k}$ and $b = (b_{-n})_{| \ k}$. Now, let $x$ be a solution to $ax = b$. The maps $s \mapsto s_{-n}$ and $s \mapsto s_{| \ k}$ are semigroup homomorphisms on the multiplicative semigroup of $S$, that is: for all $s, t \in S$ we have $( (s t)_{-n})_{| \ k} = (s_{-n})_{| \ k} (t_{-n})_{| \ k}$. In particular, 
\[
    a(x_{-n})_{| \ k} = (a_{-n})_{| \ k} (x_{-n})_{| \ k} = ( (ax)_{-n} )_{| \ k} = ( b_{-n} )_{| \ k} = b.
\]
\end{proof}

We now give the main characterisation of the set of solutions to the equation $ax = b$ in $S$.

\begin{theorem} \label{theorem:ax=b}
Let $a, b, x \in S$. For all $i \in \N$, let $l_i = b_i + a_0 b_i + a_i b_0$ and $u_i = l_i + a_i + 1$ and let
\begin{align*}
    \lambda_0 &= \bigjoin_{i \in \N} l_i \in S_0, \\
    \upsilon_0 &= \bigmeet_{i \in \N} u_i \in S_0.
\end{align*}
For all $i \ge 1$, let
\begin{align*}
    \lambda_i &= a_0 b_i + a_i b_0, \\
    \upsilon_i &= a_0 b_i + a_i b_0 + a_0 + 1.
\end{align*}

Then $ax = b$ if and only if $x_i \in [ \lambda_i, \upsilon_i ]$ for all $i \in \N$. 
Moreover, the set $\{ x \in S: ax = b \}$ is non-empty if and only if $\lambda_0 \le \upsilon_0$.
\end{theorem}

The proof of Theorem \ref{theorem:ax=b} is based on the following lemma.

\begin{lemma} \label{lemma:ax=b}
For any $a,b,x \in S$, we have $ax = b$ if and only if the following system of equations holds:
\begin{align}
    \label{equation:a0x0}
    a_0 x_0 &= b_0 \\
    \label{equation:aix0}
    a_i x_0 &= b_i + b_i a_0 + a_i b_0 \text{ for all } i \ge 1\\
    \label{equation:a0xi}
    a_0 (x_i + b_i) &= a_i b_0 \text{ for all } i \ge 1.
\end{align}
\end{lemma}

\begin{proof}
First, suppose that $x$ satisfies Equations \eqref{equation:a0x0}, \eqref{equation:aix0} and \eqref{equation:a0xi}, and let $c = ax$. By equation \eqref{equation:a0x0} we have $c_0 = a_0 x_0 = b_0$ and it follows from equations \eqref{equation:aix0} and \eqref{equation:a0xi} that for all $i \ge 1$,
\[
    c_i = a_0 x_i + a_i x_0 = (a_i b_0 + a_0 b_i) + (b_i + b_i a_0 + a_i b_0) = b_i.
\]
Thus $c = b$.

Conversely, suppose that $ax = b$. We then have $b_0 = a_0 x_0$, giving equation \eqref{equation:a0x0}, and for all $i \ge 1$, $b_i = a_i x_0 + a_0 x_i$. Since $a_0^2 = a_0$ and $a_0 x_0 = b_0$, the latter implies $a_0 x_i = a_0 b_i + a_i b_0$, which yields $a_0 (x_i + b_i) = a_i b_0$ for all $i \geq 1$, which is equation \eqref{equation:a0xi}. Thus,
\[
    a_i x_0 = b_i + a_0 x_i = b_i + a_0 b_i + a_i b_0,
\]
which is equation \eqref{equation:aix0}, and we are done.
\end{proof}

\begin{proof}[Proof of Theorem \ref{theorem:ax=b}]
The three (systems of) equations in Lemma \ref{lemma:ax=b} can be solved using Proposition \ref{proposition:ax=b_S0}: we obtain
\begin{align}
    \label{equation:a0x0_solved}
    x_0 &\in [b_0, a_0 + b_0 + 1] \\
    \label{equation:aix0_solved}
    x_0 &\in [b_i + b_i a_0 + a_ib_0, b_i + b_i a_0 + a_ib_0 + a_i + 1] \text{ for all } i \ge 1\\
    \label{equation:a0xi_solved}
    x_i + b_i &\in [a_i b_0, a_i b_0 + a_0 + 1] \text{ for all } i \ge 1.
\end{align}
Noting that $l_i = b_i + a_0 b_i + a_i b_0$ and $u_i = l_i + a_i + 1$ for all $i \ge 0$, we obtain
\begin{align*}
    x_0 &\in [l_i, u_i] \ \forall i \ge 0 \\
    x_i + b_i &\in [a_i b_0, a_i b_0 + a_0 + 1] \ \forall i \ge 1.
\end{align*}
The first system of equations is equivalent to $x_0 \in [ \lambda_0, \upsilon_0 ]$. The second system is equivalent to Equation \eqref{equation:a0xi}, which in turn is equivalent to $a_0 x_i = a_0b_i + a_i b_0$; according to Proposition \ref{proposition:ax=b_S0}, this is equivalent to $x_i \in [a_0b_i + a_i b_0, a_0b_i + a_i b_0 + a_0 + 1] = [\lambda_i, \upsilon_i]$.

Finally, if $\lambda_0 \not\le \upsilon_0$, then the set of solutions is empty. Otherwise, we have $l_i \le u_i$ for all $i \in \N$ (and hence the first system has a solution), and moreover $a_i b_0 \le a_0$ (and hence the second system has a solution) since we have $l_0 \le u_0$, hence $b_0 \le b_0 + a_0 + 1$, and equivalently $b_0 \le a_0$. 
\end{proof}

We remark that deciding whether $ax = b$ has a solution can be verified by computing $\lambda_0$, $\upsilon_0$, and verifying that $\lambda_0 \upsilon_0 = \lambda_0$. This requires $\Theta(n)$ sums and products, where $n = \max\{ N(a),  N(b) \}$ satisfies $n \le \log \max \{ |a|, |b| \}$.

\paragraph{The equation $az=0$.} 
In particular, the solutions to the equation $az = 0$ can be neatly characterised. For each $x \in S$ with $x=\sum_{i=0}^k x_i C_{2^i}$ where each $x_i$ is idempotent, define
\[
    x^+ = \bigvee_{i=0}^k x_i \in S_0. 
\]
It is clear from the definition that for all $x \in S$ and all $e \in S_0$, $e^+ = e$, $x^+ x = x$, and $e x^+ = (ex)^+$.

\begin{corollary} \label{corollary:az=0}
Let $a,z \in S$. Then $az = 0$ if and only if $z_0 a^+ = a_0 z^+ = 0$, or equivalently $z_0 \le a^+ + 1$ and $a_0 \le z^+ + 1$.
\end{corollary}

\begin{proof}
Using Lemma \ref{lemma:ax=b}, we have $az = 0$ if and only if the following system of equations holds:
\begin{align}
    \label{eq:a0z0}
    a_0 z_0 &= 0 \\
    \label{eq:aiz0}
    a_i z_0 &= 0 \text{ for all } i \ge 1\\
    \label{eq:a0zi}
    a_0 z_i &= 0 \text{ for all } i \ge 1.
\end{align}
Equations \eqref{eq:a0z0} and \eqref{eq:aiz0} are equivalent to the single equation $a^+ z_0 = 0$,
or in other words $z_0 \le a^+ + 1$. Similarly, Equations \eqref{eq:a0z0} and \eqref{eq:a0zi} are equivalent to $a_0 \le z^+ + 1$.
\end{proof}

\subsection{Extended Green's relations} \label{subsection:Green}

We next examine the multiplicative structure of $S$. Since $(S, \cdot)$ is a commutative semigroup, we have that all of Green's relations \cite{How95} coincide, that is, $\mathcal{R}=\mathcal{L}=\mathcal{H}=\mathcal{D} = \mathcal{J}$. We shall refer to this relation as $\mathcal{R}$ in all that follows, recalling that for $a,b \in S$ we have $a \mathcal{R} b$ if and only if the principal (right) ideal generated by $a$ is equal to the principal (right) ideal generated by $b$. The next proposition gives a simple characterisation of this relation, and also the extended Green relations $\mathcal{R}^*$ and $\tilde{\mathcal{R}}$. 
Recall that  for a monoid $M$, and elements $a,b \in M$ we have $a \mathcal{R}^*b$ if for all $x, y \in M$ $xa = ya$ if and only if $xb = yb$, and $a \mathcal{ \tilde{R} } b$ if for all idempotents $e$, $ea = a$ if and only if $eb = b$.

\begin{proposition}
\label{prop:green}
Let $a,b \in S$. Then  in the multiplicative semigroup of $S$ we have:
\begin{enumerate}
\item \label{item:R} $a \mathcal{R} b$ if and only if $b=ag$ for some unit $g$;

\item \label{item:Rstar} $a \mathcal{R}^* b$ if and only if $a^+ = b^+$ and $a_0=b_0$;

\item \label{item:Rtilde} $a \tilde{\mathcal{R}} b$ if and only if $a^+ = b^+$.
\end{enumerate}
\end{proposition}

\begin{proof}
\begin{enumerate}
    \item \label{item:proof_R}
    Let $G$ denote the group of units of $S$. By part (4) of Lemma \ref{lem:basis}, $G = \{g \in S: g_0=C_1\}$. Clearly if $a=bg$ for some $a,b \in S$ and $g \in G$ we have that $a \mathcal{R}b$. Suppose then that $a \mathcal{R}b$, that is $ax=b$ and $by=a$ for some $x,y \in S$. 
 
We then have $a_0 x_0 = b_0$ and  $b_0 y_0 = a_0$, giving $b_0 \leq a_0$ and $a_0 \leq b_0$ and hence $a_0=b_0$. In particular, $a_0x_0=a_0y_0=a_0$. Also, since $a=by=axy$, in the notation of subsection \ref{subsection:S_basic_properties} we have
\[
    a_+ = (axy)_+ = a_0 x_0 y_+ + a_0 x_+ y_0 + a_+ x_0 y_0 = a_0 y_+ + a_0 x_+ + a_+ x_0 y_0.
\]
In particular,
\[
    a_+ x_0 = a_0 x_0 y_+ + a_0 x_0 x_+ + a_+ x_0 y_0 = a_0 y_+ + a_0 x_+ + a_+ x_0 y_0 = a_+,
\]
and hence $ax_0 = a$.
Now consider $g = C_1 + x_+ \in G$. We have
\[
    ag = a(x_0 + C_1 + x) = ax = b,
\]
as required.

    \item \label{item:proof_Rstar}
Since $S$ is a ring we have $a \mathcal{R}^* b$ in the multiplicative semigroup of $S$ if and only if, for all $z \in S$, $az = 0$ if and only if $bz = 0$. If $a_0 = b_0$ and $a^+ = b^+$, then Corollary \ref{corollary:az=0} immediately shows that $a \mathcal{R}^* b$. Conversely, if $a^+ \ne b^+$, then there exists $z_0 \in (a^+ + 1)^\downarrow \Delta (b^+ + 1)^\downarrow$, say $z_0 \le a^+ + 1$ but $z_0 \not\le b^+ + 1$. We have (by assumption) $(z_0)^+ = z_0 \le a^+ + 1$ and hence also (by Lemma \ref{lem:basis}) $(z_0)^+ + 1  \geq  a^+ \geq a_0$. It now follows from Corollary \ref{corollary:az=0} that $a z_0 = 0$, whilst $b z_0 \ne 0$.

    \item \label{item:proof_Rtilde}
It is easy to check that the following are equivalent for any idempotent $e$: $e$ fixes $a$; $e a_i = a_i$ for all $i \in \N$; $e \ge a_i$ for all $i$; $e \ge \bigjoin_i a_i = a^+$. Thus, $a \tilde{\mathcal{R}} b$ if and only if $(a^+)^\uparrow = (b^+)^\uparrow$ in the Boolean algebra $S_0$, or equivalently $a^+ = b^+$.
\end{enumerate}

\end{proof}

We next characterise the $\{ \mathcal{R}, \mathcal{R}^*, \tilde{\mathcal{R}}\}$-class of idempotent elements. 

\begin{corollary} \label{corollary:Extended_Green_e}
Let $G$ be the group of units of $S$, and let $e \in S_0$ be idempotent.
\begin{enumerate}
    \item \label{item:R_e}
    The $\mathcal{R}$-class of $e$ is the group $eG$.

    \item \label{item:R_star_e}
    The $\mathcal{R}^*$-class of $e$ is also the group $eG$.

    \item \label{item:R_tilde_e}
    The $\tilde{\mathcal{R}}$-class of $e$ is the set $\{ a \in S : a^+ = e\} \subseteq eS$.

\end{enumerate}
\end{corollary}

\begin{proof}
\begin{enumerate}
    \item \label{item:proof_R_e}
    This follows directly from Proposition \ref{prop:green}.

    \item \label{item:proof_R_star_e}
    Suppose $a \mathcal{R}^* e$. By Proposition \ref{prop:green}, we have $a^+ = e$ and $a_0 = e$, or equivalently $a_0 = e$ and $a_i e = a_i$ for all $i \ge 1$. We obtain $a = eg$ with $g_0 = C_1$ and $g_i = a_i$ for all $i \ge 1$.

    \item \label{item:proof_R_tilde_e}
    We have $a \tilde{R} e$ if and only if $a^+ = e$. Since $a^+ a = a$, this implies $a = ea$.
We note that for any idempotent $e \in S$, the ideal $eS$ is the set of elements fixed by $e$.
\end{enumerate}

\end{proof}

Recall that a commutative semigroup $T$ is said to be a \Define{restriction semigroup} if (1) for any
$a \in T$, $a$  is $\tilde{\mathcal{R}}$-related to a (unique) idempotent $a^+$ and (2) for any element $a$ and any idempotent $e$ of $T$ we have $ae = (ae)^+ a$. For more on restriction semigroups, the reader is directed to the papers \cite{GH10, Kud25} and the notes by Gould in \cite{Gou07}.

\begin{proposition} \label{proposition:restriction_semigroup}
The multiplicative semigroup of $S$ is a restriction semigroup with respect to $a^+$, with $\mathcal{R} \subsetneq {\mathcal{R}}^* \subsetneq \tilde{\mathcal{R}}$.
\end{proposition}

\begin{proof}
It is clear that $S$ satisfies the first condition where for $x \in S$ the unique idempotent related to $x$ is $x^+$. For the second condition, we have $(ae)^+ = e a^+$ (as previously shown) and $(ae)^+ a = e a^+ a = e a$. 

We now give examples to show the two strict inclusions. Firstly, if $a = C_1$ and $b = C_2$, we have $a^+ = b^+ = C_1$ but $a_0 \ne b_0$, and hence $a \tilde{\mathcal{R}} b$ but not $a \mathcal{R}^* b$. Secondly, let $a = C_3 + C_5 C_2$ and $b = C_3 + C_5 C_4$. Then clearly $a \mathcal{R}^* b$; we now prove that $a$ and $b$ are not $\mathcal{R}$-related. Suppose, for the sake of contradiction, that $a = bg$ for some unit $g$. We have 
\[
    C_5 = a_1 = b_1 g_0 + b_0 g_1 = C_3 g_1,
\]
and hence $C_5 \le C_3$, thus $C_3C_5 = C_5$. However, $C_3C_5 = C_{15}$, which is the desired contradiction.
\end{proof}

\subsection{Special elements of $S$} \label{subsection:special_elements}

We have already characterised the idempotents of $S$. In this subsection, we characterise the group of units, the regular elements (and all the subgroups of $S$), and the so-called co-regular elements of $S$.

An element $a \in S$ is (multiplicatively) \Define{cancellable} if $ax = ay$ implies $x = y$.

\begin{theorem}[Classification of cancellable elements of $S$] \label{theorem:cancellable_elements}
The following are equivalent for $g \in S$:
    \begin{enumerate}
        \item \label{item:finite_number_gx=0} there are finitely many solutions $x \in S$ to the equation $gx = 0$;

        \item \label{item:unique_gx=0} $x = 0$ is the unique solution of $gx = 0$;

        \item \label{item:g_cancellable} $g$ is cancellable;

        \item \label{item:g_unit} $g$ is a unit;

        \item \label{item:g0=1} $g_0 = 1$;

        \item \label{item:g=a2+a+1} $g = a^2 + a + 1$ for some $a \in S$;
        \item \label{item:order2}$g^2=1$.
    \end{enumerate}
\end{theorem}

\begin{proof}
Equivalence of items \ref{item:unique_gx=0} and \ref{item:g_cancellable} follows from the fact that $S$ is a ring (since $ax=ay$ if and only if $a(x-y) = 0$). For the equivalence of items  \ref{item:g_cancellable} and \ref{item:g_unit}, note that $g$ is cancellable if and only if $g \mathcal{R}^* 1$, which by Proposition \ref{prop:green} is the case if and only if $g$ is in the group of units. 
We have already proved the equivalence of items \ref{item:g_unit}, \ref{item:g0=1} and \ref{item:order2} in Lemma \ref{lem:basis}. To see the equivalence of $\ref{item:g=a2+a+1}$ and $\ref{item:g0=1}$, first note that if $g = a^2 + a + 1$, then $g_0 = g^2 = a^4 + a^2 + 1 = 1$. Conversely, if $g_0 = 1$, then $g = 0^2 + 0 + 1$. The implication  \ref{item:unique_gx=0} implies \ref{item:finite_number_gx=0} is trivial.

To complete the proof we show that    \ref{item:finite_number_gx=0} implies \ref{item:g0=1}. Suppose then that there are finitely many solutions to the equation $gx=0$ and suppose for contradiction that $g = \sum_i=1^t g_i C_{2^i}$ where $g_0 \ne 1$. For each $x^{(k)} = (g_0 + 1) C_{ 2^k }\in S_+$ where $k \geq 1$, recalling that $s_+t_+ = 0$ for all $s,t \in S$, we have 
$$g x^{(k)} = (g_0 + g_+)(g_0+1)C_{2^k} = g_0(g_0+1)C_{2^k}  + 0 =    (g_0+g_0)C_{2^k}=0,$$
contradicting that there are finitely many solutions of $ax = 0$.

\end{proof}

The group of units, given by $G = C_1 + S_+$, is isomorphic to $( \F_2^\N, + )$: for all $g, g' \in G$, we have $(gg')_0  = C_1$ and for all $i \ge 1$,
\[
    (gg')_i = g_0 g'_i + g_i g'_0 = g_i + g'_i = (g + g')_i.
\]
Therefore, $(gg') + C_1 = (g + C_1) + (g' + C_1)$. In particular, any unit is an involution: $g^2 = 1$. In fact, every subgroup (or equivalently, regular $\mathcal{H}$-class) is the semigroup $eG$, where $e \in S_0$ is idempotent. Then $eG = e + eS_+$ again satisfies $(eg eg') + e = (eg + e) + (eg' + e)$; thus $eG$ is isomorphic to $( \F_2^\N, + )$, unless $e = 0$, in which case it is the trivial group.

\begin{theorem}[Classification of regular elements]     \label{theorem:classification_regular_elements}
The following are equivalent for $a \in S$:
\begin{enumerate}
    \item \label{item:regular_aRa2} $a \mathcal{R} a^2$;

    \item \label{item:regular_a2g} $a = a^2 g$ for some unit $g \in G$;

    \item \label{item:regular_a3=a} $a^3 = a$;

    \item \label{item:regular_a=s3} $a = s^3$ for some $s \in S$;

    \item \label{item:regular_a0=sigma} $a^+ = a^2$;

    \item \label{item:regular} $a$ is regular;

    \item \label{item:regular_a_subgroup} $a$ belongs to a subgroup;

    \item \label{item:regular_aRe} $a \mathcal{R} e$ for some idempotent $e \in E(S)$;

    \item \label{item:regular_aRstara2} $a \mathcal{R}^* a^2$;
    
    \item \label{item:regular_aRstare} $a \mathcal{R}^* e$ for some idempotent $e \in E(S)$.
\end{enumerate}
\end{theorem}

\begin{proof}
Firstly, the equivalence between items \ref{item:regular}, \ref{item:regular_a_subgroup}, \ref{item:regular_aRe} is guaranteed by the fact that $S$ is commutative. Secondly, the equivalence between items \ref{item:regular_aRa2} and \ref{item:regular_a2g} is given in Proposition \ref{prop:green}. Thirdly, the following implications are immediate from the definitions: item \ref{item:regular_a2g} implies item \ref{item:regular_aRe} (recalling that $a^2$ is idempotent), item \ref{item:regular_a3=a}  implies item \ref{item:regular_a=s3} which in turn implies item \ref{item:regular},  item \ref{item:regular_aRa2} implies item \ref{item:regular_aRstara2}, which in turn implies \ref{item:regular_aRstare}. To complete the proof we show: \ref{item:regular_a0=sigma} if and only if \ref{item:regular_a3=a}, \ref{item:regular_aRe} implies  \ref{item:regular_a2g};  \ref{item:regular_a2g} implies \ref{item:regular_a3=a}; \ref{item:regular_a=s3} implies \ref{item:regular_a3=a}; and  \ref{item:regular_aRstare} implies \ref{item:regular_a_subgroup}.

To see that \ref{item:regular_a0=sigma} implies \ref{item:regular_a3=a}: if $a^+ = a^2$, then $a^3 = a a^+ = a$.

To see that \ref{item:regular_a3=a} implies \ref{item:regular_a0=sigma}: if $a^3 = a$, then let $b = a^3$; we have $a_i = b_i = a_ia_0$ for all $i \in \N$, or equivalently $a_0 \ge a^+$ (and by definition, $a_0 \le a^+$).

To see that \ref{item:regular_aRe} implies  \ref{item:regular_a2g}: by Proposition \ref{prop:green}, if $a \mathcal{R} e$ then   $a = eg$ for some unit $g \in G$, then we have $a^2 = a_0 = (eg)_0 = e g_0 = e$, and hence $a^2g = eg = a$.

To see that \ref{item:regular_a2g} implies \ref{item:regular_a3=a}: if $a = a^2g$ for some unit $g$, then $a^3 = a^2a^2g = a^4 g = a^2g=a$.

To see that \ref{item:regular_a=s3} implies \ref{item:regular_a3=a}:  if $a = s^3$, then using the fact that $s^4 = s^2$, $a^3 = s^9 = s^3 = a$.

To see that  \ref{item:regular_aRstare} implies \ref{item:regular_a_subgroup}: if $a \mathcal{R}^* e$, then $a$ belongs to the group $eG$ by Corollary \ref{corollary:Extended_Green_e}.
\end{proof}

Say an element $h \in S$ is \Define{co-regular} if $h + C_1$ is regular. We first give a characterisation of co-regular elements.

\begin{proposition}[Classification of co-regular elements] \label{proposition:co-regular}
Then the following are equivalent for $h \in S$:
\begin{enumerate}
    \item \label{item:co-regular} $h$ is co-regular, i.e. $h + C_1$ is regular;

    \item \label{item:co-regular_h3=h2} $h^3 = h^2$;

    \item \label{item:co-regular_h0hi} $h_0 h_i = 0$ for all $i \ge 1$;

    \item \label{item:co-regular_h(a)} $h = a^3 + a^2 + a$ for some $a \in S$;

    \item \label{item:co-regular_h=b3} $h = b^3 + 1$ for some $b \in S$.
\end{enumerate}
\end{proposition}

\begin{proof}
The equivalence of \ref{item:co-regular} and \ref{item:co-regular_h=b3} follows immediately from Theorem \ref{theorem:classification_regular_elements}.

To see that \ref{item:co-regular_h=b3} implies \ref{item:co-regular_h3=h2}: If $h = b^3 + 1$, we have $h^2 = b^6 + 1 = b^2 + 1$, and 
\[
    h^3 = (b^3 + 1)(b^2 + 1) = b^5 + b^3 + b^2 + 1 = b^3 +b^3 +b^2 + 1 =b^2 + 1= h^2.
\]

To see that \ref{item:co-regular_h3=h2} implies \ref{item:co-regular}: if $h^3 = h^2$, then $(h+1)^3 = h^3 + h^2 + h + 1 = h + 1$, hence $h + 1$ is regular.

The equivalence of \ref{item:co-regular_h(a)} and \ref{item:co-regular_h=b3} follows from the fact that $a^3 + a^2 + a = (a+1)^3 + 1$.

For the equivalence of \ref{item:co-regular_h3=h2} and \ref{item:co-regular_h0hi}, notice that the equation $h^3 = h^2$ is equivalent to $h_0h = h_0$, which in turn is equivalent to $h_0 h_i = 0$ for all $i \ge 1$.
\end{proof}

\begin{corollary}
For any $e \in S$, $e$ is idempotent if and only if $e$ is regular and co-regular.
\end{corollary}

We show next that the co-regular elements provide a transversal of the $\mathcal{R}$-classes. 

\begin{proposition} \label{proposition:co-regular_R-class}
Every $\mathcal{R}$-class of $S$ contains a unique co-regular element. Specifically, let $H: S \rightarrow S$ be given by $H(a) = a + a^2 + a^3$. Then for all $a \in S$, $H(a)$ is the unique co-regular element in the $\mathcal{R}$-class of $a$, and hence in particular $H$ is constant on $\mathcal{R}$-classes.
\end{proposition}

\begin{proof}
For each $a \in S$, we have $H(a) = a(1 + a + a^2) = a g_a$, where $g_a = 1 + a + a^2$ is a unit (by Theorem \ref{theorem:cancellable_elements}). Thus by Proposition \ref{prop:green} we have $H(a) \mathcal{R} a$. Since (by Proposition \ref{proposition:co-regular}) $H(a)$ is co-regular, this guarantees the existence of a co-regular element in each $\mathcal{R}$-class.

Now suppose $f \in S$ is co-regular and $\mathcal{R}$-related to $a$. By Proposition \ref{prop:green} and the previous paragraph we have $f = H(a) g$ for some unit $g$. Recalling that $H(a) = ag_a$ where $g_a$ is a unit and that $h^2=1$ for all units $h$ (by Theorem \ref{theorem:cancellable_elements}) we find $f^2 = H(a)^2g^2 = a^2g_a^2g^2 = a^2$  and $f^3 = f^2f = a^2H(a)g = (a^3+a^4+a^5)g = a^3g + a^2g + a^3g = a^2g$. Since $f$ is co-regular, Proposition \ref{proposition:co-regular} now gives $a^2g= a^2$. This in turn, together with the fact that $g_0=1$, implies that 
$$ag = (a_0 + a_+)g = (a^2+a_+)g = a^2g + a_+g = a^2 + a_+g = a^2 + a_+(1+g_+)  = a^2 + a_+ +0= a_0 + a_+= a,$$ from which we deduce that $g=1$ and hence $f = H(a)$. This guarantees the $H(a)$ is the unique co-regular element in the $\mathcal{R}$-class of $a$. Notice that for all $b$ with $a\mathcal{R} b$ we then have $H(b) = H(a)$.
\end{proof}

We obtain an alternative characterisation of $\mathcal{R}$.

\begin{corollary} \label{corollary:R}  
For any $a, b \in S$, $a \mathcal{R} b$ if and only if $a_0 = b_0 \ge (a+b)^+$.
\end{corollary}

\begin{proof}
If $a \mathcal{R} b$, then $a \mathcal{R}^* b$ and hence $a_0 = b_0$. Thus let us assume from now on that $a, b \in S$ and $a_0=b_0$ and aim to show that for such $a,b \in S$ we have $a \mathcal{R} b$ if and only if $a_0 \geq (a+b)^+$.

By Proposition \ref{proposition:co-regular_R-class} together with the fact that $x^2 = x_0$ for all $x \in S$ we see that $a \mathcal{R} b$ if and only if $a + a_0 + a_0a = b + b_0 + b_0b$, or equivalently (since $a_0=b_0$) $a + a_0a = b + a_0b$. This in turn is equivalent (working, as always, over $\mathbb{F}_2$) to $(a_0+1)(a+b) = 0$. By Corollary \ref{corollary:az=0} the latter is equivalent to
$a_0 + 1 \leq (a+b)^+ + 1$ and $a_0+b_0 \leq (a_0 + 1)^+$. The second of these inequalities holds trivially since $a_0=b_0$, whilst by Lemma \ref{lem:basis} we have that $a_0+1 \leq (a+b)^+ + 1$ is equivalent to $a_0 \geq (a+b)^+$.
\end{proof}

\subsection{Injective polynomials} \label{subsection:injective_polynomials}

In this subsection, we classify the injective univariate polynomials over $S$, which turn out to also be the surjective univariate polynomials over $S$. We show that solving polynomial equations of the form $P(x) = e$, where $P$ is bijective and $s \in S$ is fixed, is straightforward. First of all, we remark that, since $x^4 = x^2$ for all $x \in S$, we can assume that $P(x)$ is at most cubic; as such, we consider $P(x) = a x^3 + b x^2 + c x + d$.

\begin{theorem} \label{theorem:injective_polynomials}
Let $P(x) = a x^3 + b x^2 + c x + d$ be a polynomial in $S[x]$. The following are equivalent.
\begin{enumerate}
    \item $P$ is injective;

    \item $P$ is surjective;

    \item $(a_0, b_0, c_0) = (0,0,1)$.
\end{enumerate}
Moreover, if $P$ is bijective, then for all $s \in S$,
\[
    P(x) = s \quad \iff \quad x_0 = d_0 + s_0 \quad \text{and} \quad x_+ = d_+ + s_+ + (a + b + c)_+ (d_0 + s_0).
\]
\end{theorem}

\begin{proof}
Firstly, suppose $(a_0 , b_0, c_0) = (0,0,1)$. The element $g = (a + b + c)_+ + 1$ is a unit, and we have
\[
    P(x) = a x^3 + b x^2 + c x + d = a_+ x_0 + b_+ x_0 + c_+ x + x + d = g x_0 + x_+ + d.
\]
Therefore, for any $s \in S$, we have
\[
    P(x) = s \iff g x_0 + x_+ = d + s,
\]
which yields $(g x_0 + x_+)_0 = x_0 = d_0 + s_0$ and $(g x_0 + x_+)_+ = g_+ (d_0 + s_0) + x_+ = d_+ + s_+$, thus $x_+ = d_+ + s_+ + (a + b + c)_+ (d_0 + s_0)$. (And in particular, $P$ is bijective.)

The rest of the proof is devoted to proving that an injective, or surjective, polynomial $P$ has $(a_0, b_0, c_0) = (0,0,1)$. Clearly, the constant term $d$ has no effect on injectivity or surjectivity, hence we assume $d = 0$ henceforth.

Suppose first that $P = ax^3+bx^2+c$ is injective aiming to show that $(a_0, b_0, c_0) = (0,0,1)$. Suppose for contradiction $(a_0, c_0) \ne (0,1)$.  Simple calculations yield
\begin{align}
    \label{equation:Px0}
    P(x)_0 &= (a + b + c)_0 x_0, \\
    \label{equation:Px+}
    P(x)_+ &= (a_0 x_0 + c_0) x_+ + (a + b + c)_+ x_0.
\end{align}
We note that for any $\alpha_0, \gamma_0 \in S_0$, we have $\alpha_0 x_0 + \gamma_0 = 1$ for all $x_0 \in S_0$ if and only if $(\alpha_0, \gamma_0) = (0,1)$. Indeed setting $x_0 = 0$ forces $\gamma_0 = 1$ and then setting $x_0 = 1$ gives $1 = \alpha_0 + \gamma_0 = \alpha_0 + 1$ giving $\alpha_0 = 0$. Since $(a_0, c_0) \ne (0,1)$, we conclude that there exists $x_0$ such that $\mu_0 = a_0 x_0 + c_0 \ne 1$. Let 
\[
    x = x_0, \quad y = x_0 + (\mu_0 + 1) C_2.
\]
Since $x_0 = y_0$, we have $P(x)_0 = P(y)_0$ by Equation \eqref{equation:Px0}, and 
\[
    P(y)_+ = (a_0 x_0 + c_0) (\mu_0 + 1) C_2 + (a + b + c)_+ x_0 = \mu_0 (\mu_0 + 1) C_2 + P(x)_+ = P(x)_+
\]
by Equation \eqref{equation:Px+}. Thus, $P$ is not injective. Next, suppose $(a_0, c_0) = (0,1)$ but $b_0 \ne 0$. Again, we obtain 
\begin{align*}
    P(x)_0 &= (b_0 + 1) x_0, \\
    P(x)_+ &= x_+ + (a + b + c)_+ x_0.
\end{align*}
Therefore, if $x = 0$ and $y = b_0 + (a + b + c)_+ b_0$, we obtain $P(x) = P(y) = 0$. Thus, $P$ is not injective once more.

Finally, we prove that if $P$ is surjective, then $(a_0, b_0, c_0) = (0,0,1)$. We consider $x, y \in S$ such that
\begin{align*}
    P(x) &= C_1 + C_2 + (a + b + c)_+ \\
    P(y) &= C_2.
\end{align*}
First, $P(x)_0 = (a + b + c)_0 x_0 = 1$, which implies $x_0 = 1$ and $(a + b + c)_0 = 1$. Second, $P(y)_0 = ( a + b + c )_0 y_0 = y_0 = 0$, and hence $P(y)_+ = c_0 y_+ = C_2$, which implies $c_0 = 1$. Third, $P(x)_+ = (a_0 + 1) x_+ + (a + b + c)_+ = C_2 + (a + b + c)_+$, and hence $(a_0 + 1)x_+ = C_2$, which implies $a_0 = 0$. We obtain $(a_0, b_0, c_0) = (0,0,1)$.
\end{proof}

\subsection{Intersection of principal ideals} \label{subsection:intersection_principal_ideals}

In this subsection, we are interested in the intersection $xS \cap yS$ of two principal ideals. The main question is whether this intersection remains a principal ideal. We are not able to fully answer this question, but we prove that under natural constraints, $xS \cap yS$ is indeed a principal ideal. On the other hand, we have not yet found any example of $x, y \in S$ such that $xS \cap yS$ is not a principal ideal.

We first prove that we can focus on the case where $x \tilde{\mathcal{R}} y$.

\begin{lemma} \label{lemma:alpha_beta}
Let $x,y \in S$. Denoting $\alpha = x y^+$ and $\beta = y x^+$, we have $xS \cap yS = \alpha S \cap \beta S$ and $\alpha \tilde{\mathcal{R}} \beta$.
\end{lemma}

\begin{proof}
Set $\alpha = x y^+$ and $\beta = y x^+$.
It is then clear that $\alpha S  = x y^+ S \subseteq xS$ and $\beta S  = y x^+ S \subseteq yS$ and hence $\alpha S\cap \beta S \subseteq xS \cap yS$. Furthermore,
$\alpha^+ = x^+ y^+ = \beta^+$.  Now suppose that $m = xa = yb \in xS \cap yS$. 
Then $m = xa = x^+ xa = x^+ m = x^+ y b = \alpha b$ and by symmetry $m = \beta a$, giving $m \in \alpha S \cap \beta S$.  
\end{proof}

\begin{proposition} \label{proposition:x_regular}    
Let $x,y \in S$ with $x$ regular. Then $xS \cap yS = xyS$.
\end{proposition}

\begin{proof}
If $x$ is regular, then (by Theorem \ref{theorem:classification_regular_elements}) $xS = x^2S$ where $x^2$ is idempotent, so without loss of generality suppose $x$ is idempotent giving $x^+ = x$. Using the notation from Lemma \ref{lemma:alpha_beta}, we have
\[
    \beta = y x^+ = yx = y y^+ x = \alpha y,
\]
thus $\beta \in \alpha S$ and $xS \cap yS = \alpha S \cap \beta S = \beta S= xyS$.
\end{proof}

One can easily construct examples in which $xS \cap yS \neq xyS$; for instance if $x = y=C_2$, then $C_2 \in C_2S = xS \cap yS \neq \{0\}$, whilst $xyS = \{0\}$, since $xy=0$. This is a special case of the following result.

\begin{proposition} \label{proposition:xy=0}
If $xy = 0$, then $xS \cap yS = \gamma_{x,y}S$, where $\gamma_{x,y} = (1 + (x+y)^+ )x = (1 + (x+y)^+ )y$.
\end{proposition}

\begin{proof}
Suppose now that $xy=0$, so that
\begin{equation}
\label{eq:xy=0}
x_0y_0=0, \mbox{ and } x_0y_i+y_0x_i=0 \mbox{ for all } i>0,
\end{equation}
and consider $xS \cap yS$. By Lemma \ref{lemma:alpha_beta} and Proposition \ref{prop:green} we may assume without loss of generality that $x^+ = y^+$. It follows from the fact that $(x+y)^+ (x+y) = x+y$ that $(1 + (x+y)^+)x = (1 + (x+y)^+ )y$. Calling this element $\gamma_{x,y}$ it is clear that $\gamma_{x,y}S \subseteq xS \cap yS$. We aim to show the converse. Suppose then that $m = xa=yb \in xS \cap yS$ so that
\begin{equation}
\label{eq:m}
m_0=x_0a_0=y_0b_0, \mbox{ and } m_i= y_0b_i+b_0y_i=x_0a_i+a_0x_i \mbox{ for all } i>0.
\end{equation}
Idempotency of $x_0$ together with Equations \eqref{eq:xy=0} and \eqref{eq:m} then gives \begin{eqnarray*}
m_0 &=& x_0a_0 = x_0^2a_0 = x_0y_0b_0 = 0,\\
x_0m_i &=& x_0^2a_i+x_0a_0x_i = x_0a_i,  \mbox{ and}\\   
x_0m_i &=&x_0y_0b_i + x_0b_0y_i = 0 + y_0b_0x_i = 0.
\end{eqnarray*}
Thus $x_0m_i=x_0a_i=0$ for all $i$, and a dual argument (this time using idempotency of $y_0$) also shows that $y_0m_i=y_0b_i=0$ for all $i$. But then by Equation \eqref{eq:m} we have $m_0= a_0x_0=b_0y_0$ and
$m_i = x_0a_i + a_0x_i = a_0x_i$ and $m_i = y_0b_i + b_0y_i = b_0y_i$ for all $i>0$, together giving $$m=a_0x =b_0y.$$
Now, recalling that $x^+ = y^+$, we have
\[
    a_0 x^+ = (a_0 x)^+ = m^+ = (b_0 y)^+ = b_0 y^+,
\]
and so
$$a_0(x+y) = a_0x+a_0y = a_0x + a_0 x^+ y = a_0x + b_0y =m+m= 0.$$
By Corollary \ref{corollary:az=0}, we obtain $a_0 (x+y)^+ = 0$, and hence
\[
    m = a_0 x = a_0 x +  a_0 (x+y)^+ x = \gamma_{x,y} a_0 \in \gamma_{x,y} S.
\]
\end{proof}

\section{The semiring of injective partial transformations} \label{section:injective}

In this section, we now consider all injective partial transformations, i.e. sums of chains 
$L_d$ and cycles $C_e$. It is straightforward to verify that (working with coefficients in $\N$) for any $e, d\geq 1$ we have
\begin{eqnarray*}
L_eL_d  &=& (|d-e|+1)L_{{\rm min}(d,e)} + \sum_{i=1}^{{\rm min}(d,e)-1} 2L_i,\\ 
C_eC_d &=& {\rm gcd}(d,e) C_{\lcm(d,e)},\\
L_eC_d &=& dL_e.
\end{eqnarray*}

Recall that the set of all chains is denoted by $L$ and that the set cycles is denoted by $C$. When working in $\F_2(L \cup C)$, we obtain
\[
    L_e L_d = \begin{cases}
        L_{\min( d,e )} & \text{if } e \equiv d \mod 2 \\
        0 &\text{otherwise}.
    \end{cases}
\]

\subsection{The lattice generated by $L$ with $\F_2$} \label{subsection:semiring_L2}

We first consider the case of sums of chains modulo $2$, i.e. $\F_2 L$. We first remark that even and odd lengths do not interact, as $L_d L_e = 0$ if $e$ and $d$ have different parities. We can then decompose our problem into a problem for the even chains and another for the odd chains. More precisely, for $\epsilon \in \{0,1\}$, let 
\[
    L^{(\epsilon)} = \left\{ a \in \F_2 L : a = \sum_{j \ge 1} \alpha_{2j - \epsilon} L_{2j - \epsilon} \right\}, 
\]
i.e. $L^{(0)}$ involves even chains only, while $L^{(1)}$ involves odd chains instead. Then any $a \in \F_2 L$ can be uniquely decomposed as $a = a^{(0)} + a^{(1)}$, where $a^{(0)} = \sum_{ j \equiv 0 \mod 2 } \alpha_j L_j \in L^{(0)}$ and $a^{(1)} = \sum_{ j \equiv 1 \mod 2 } \alpha_j L_j \in L^{(1)}$. Therefore,
\[
    a x = b \quad \iff \quad b^{(0)} = a^{(0)} x^{(0)} \text{ and } b^{(1)} = a^{(1)} x^{(1)}.
\]
As such, we consider the division problem in $L^{(1)}$ only; the case for $L^{(0)}$ is equivalent.

In $\F_2 L$, $L_d^2 = L_d$, hence the product forms a semilattice. In fact, it is isomorphic to the lattice of finite subsets of $\N$. Indeed, let $Z_1 = L_1$, and $Z_i = L_i + L_{i-2}$ for all odd $i \ge 3$. Now, for all odd $i, j$ we have: 
\[
    Z_i Z_j = \delta_{ij} Z_i,
\]
so that if $a = \sum_i a_i Z_i$ and $b = \sum_j b_j Z_j$, we have $ab = \sum_k a_k b_k Z_k$. In other words, if we identify $a$ with its support ($S(a) = \{ i : a_i = 1\}$), then $S(ab) = S(a) \cap S(b)$.
We note that, due to the absence of the multiplicative identity $C_1$, the lattice $\F_2 L$ is not a Boolean algebra. However, the sublattice generated by $\mathcal{L}_k = \{ Z_1, Z_3, \dots, Z_k\}$ is indeed a Boolean algebra $V( \mathcal{L}_k )$, with atoms $Z_1, \dots, Z_k$ and upper bound given by $Z_1 + Z_3 + \dots + Z_k = L_k$. For instance, the Boolean algebra generated by $Z_1, Z_3, Z_5$ is given in Figure \ref{figure:Z1Z3Z5}.

\begin{figure}
\centering
\begin{tikzpicture}[xscale=5, yscale=2, every node/.style=draw,thick]
    \node (000) at (0,0) {$0$};
    
    \node (100) at (-1,1) {$Z_1 = L_1$};
    \node (010) at (0,1) {$Z_3 = L_1 + L_3$};
    \node (001) at (1,1) {$Z_5 = L_3 + L_5$};

    \node (110) at (-1,2) {$Z_1 + Z_3 = L_3$};
    \node (101) at (0,2) {$Z_1 + Z_5 = L_1 + L_3 + L_5$};
    \node (011) at (1,2) {$Z_3 + Z_5 = L_1 + L_5$};

    \node (111) at (0,3) {$Z_1 + Z_3 + Z_5 = L_5$};

    \path[draw]
    (000) to (100)
    (001) to (101)
    (010) to (110)
    (011) to (111)
    (000) to (010)
    (001) to (011)
    (100) to (110)
    (101) to (111)
    (000) to (001)
    (010) to (011)
    (100) to (101)
    (110) to (111);

\end{tikzpicture}
    \caption{The Boolean algebra $V( \mathcal{L}_5 )$ generated by chains of height $1$, $3$, and $5$.} \label{figure:Z1Z3Z5}
\end{figure}
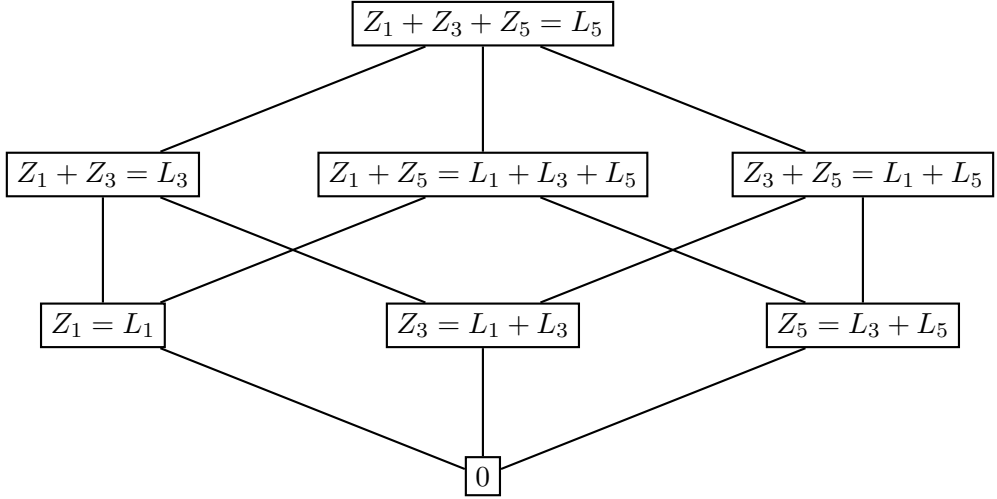

For any $a = \sum_j \alpha_j L_j$, we define the \Define{height} of $a$ as 
\[
    \Height( a ) = \max\{ j : \alpha_j = 1 \}, 
\]
with the convention that $\Height(0) = 0$. If $a, x, b \in L^{(1)}$ satisfy $a x = b$, then $\Height( b ) = \min\{ \Height( a ), \Height( x ) \}$. For any $x \in L^{(1)}$ and any $h \in \N$ odd, let
\begin{align*}
    x_{- h} &= \sum_{i \le h} x_i Z_i,\\
    x_{+ h} &= \sum_{j \ge h + 2} x_j Z_j,
\end{align*}
so that $x = x_{ -h } + x_{ +h }$.

\begin{theorem} \label{theorem:division_Lh}
Let $a, b \in L^{(1)}$ with $\Height( a ) = h$. Then $x \in L^{(1)}$ satisfies $a x = b$ if and only if $x_{ -h } \in [b, b + a + L_h]$ in the Boolean algebra $V( \mathcal{L}_h )$.
\end{theorem}

\begin{proof}
For any $k \ge h = \Height(a)$, we have $a L_k = a$ and hence $a Z_{k+2} = a ( L_{k+2} + L_k ) = a+a =0$, from which it follows that $a x_{ +h } = 0$ for any $x \in L^{(1)}$. We obtain
\[
    ax = b \iff a x_{ -h } = b,
\]
and the latter is, according to Proposition \ref{proposition:ax=b_S0}, equivalent to $x_{ -h } \in [b, b + a + L_h]$ in the Boolean algebra $V( \mathcal{L}_h )$.
\end{proof}

\subsection{Dividing in the semiring generated by $L, C$ with $\F_2$} \label{subsection:division_LC}

Any $d \in \F_2(L \cup C)$ can be uniquely decomposed as $d = d^{(0)} + d^{(1)} + d^{(C)}$, where $d^{(\epsilon)} \in L^{(\epsilon)}$ for $\epsilon \in \{0,1\}$ and $d^{(C)} \in \F_2 C = S$. In order to keep notation simple, we denote $d^{(C)} = \sum_i d_i C_{2^i}$ where $d_i \in S_0$ for all $i$.

We now give a characterisation of the set of solutions to the division equation $ax = b$ in $\F_2( L \cup C )$. First, we introduce some useful notation. For all $i \in \N$, let $\lambda_i, \upsilon_i$ be as in Theorem \ref{theorem:ax=b} for the division $a^{(C)} x^{(C)} = b^{(C)}$. For all $t \in \{0,1\}$, let $X_t$ denote the set of solutions to $a^{(C)} x^{(C)} = b^{(C)}$, where $|x^{(C)}| \equiv t \mod 2$. According to Theorem \ref{theorem:ax=b}, this set is given by
\[
    X_t = \left\{ x^{(C)} \in S : x_i \in [ \lambda_i, \upsilon_i ] \ \forall i \in \N, |x^{(C)}| \equiv t \mod 2 \right\}.
\]
For $\epsilon \in \{0,1\}$, let
\[
    h = h^{(\epsilon)} = \begin{cases}
        \Height( a^{(\epsilon)} ) &\text{ if } |a^{(C)}| \equiv 0 \mod 2, \\
        \max ( \Height( a^{(\epsilon)} ), \Height( b^{(\epsilon)} ) ) &\text{ if } |a^{(C)}| \equiv 1 \mod 2.
    \end{cases}
\]
Then let
\[
    Y_t^{(\epsilon)} = \begin{cases} 
    \{ x^{(\epsilon)} \in L^{(\epsilon)} :  x^{(\epsilon)}_{ -h } \in [ b^{(\epsilon)} + t a^{(\epsilon)}, b^{(\epsilon)} + (t + 1) a^{(\epsilon)} + L_h ] \subseteq V( \mathcal{L}_h ) \} & \text{if } |a^{(C)}| \equiv 0 \mod 2,\\
    \{ x^{(\epsilon)} \in L^{(\epsilon)} : x^{(\epsilon)} = x^{(\epsilon)}_{ -h } \in [ b^{(\epsilon)} + t a^{(\epsilon)}, b^{(\epsilon)} + (t + 1) a^{(\epsilon)} ] \subseteq V( \mathcal{L}_h ) \} & \text{if } |a^{(C)}| \equiv 1 \mod 2.
    \end{cases}
\]

\begin{theorem} \label{theorem:division_LC}
Let $a, b, x \in \F_2(L \cup C)$. Then, following the notation above, $ax = b$ if and only if $x^{(C)} \in X_t$ for some $t \in \{0,1\}$ and $x^{(\epsilon)} \in  Y_t^{(\epsilon)}$ for all $\epsilon \in \{0,1\}$.
\end{theorem}

\begin{proof}
We have $ax = b$ if and only if
\begin{align}
    \label{equation:bc}
    b^{(C)} &= a^{(C)} x^{(C)}, \\
    \label{equation:bepsilon}
    b^{(\epsilon)} &= a^{(\epsilon)} x^{(\epsilon)} + a^{(\epsilon)} x^{(C)} + a^{(C)} x^{(\epsilon)}, \forall \epsilon \in \{0,1\}.
\end{align}

Firstly, thanks to Theorem \ref{theorem:ax=b}, we can solve Equation \eqref{equation:bc}. 
As mentioned above, the set of $x^{(C)}$ solutions to Equation \eqref{equation:bc} with $|x^{(C)}| \equiv t \mod 2$ for each $t \in \{0,1\}$ is $X_t$.

Secondly, for $|x^{(C)}| \equiv t \mod 2$, Equation \eqref{equation:bepsilon} is equivalent to
\begin{equation}
    \label{equation:bepsilont}
    b^{(\epsilon)} = a^{(\epsilon)} x^{(\epsilon)} + t a^{(\epsilon)} + a^{(C)} x^{(\epsilon)}.
\end{equation}

\begin{claim}
Suppose $|a^{(C)}| \equiv 0 \mod 2$. Denoting $h = \Height( a^{(\epsilon)} )$, $x^{(\epsilon)}$ is a solution to Equation \eqref{equation:bepsilont} if and only if
\[
    x^{(\epsilon)}_{ -h } \in [ b^{(\epsilon)} + t a^{(\epsilon)}, b^{(\epsilon)} + (t + 1) a^{(\epsilon)} + L_h ]
\]
in $V( \mathcal{L}_h )$.
\end{claim}

\begin{proof}
If $|a^{(C)}| \equiv 0 \mod 2$, then Equation \eqref{equation:bepsilont} reduces to
\[
    a^{(\epsilon)} x^{(\epsilon)} = b^{(\epsilon)} + t a^{(\epsilon)},
\]
which by Theorem \ref{theorem:division_Lh}, yields $x^{(\epsilon)}_{ -h } \in [ b^{(\epsilon)} + t a^{(\epsilon)}, b^{(\epsilon)} + (t + 1) a^{(\epsilon)} + L_h ]$.
\end{proof}

\begin{claim}
Suppose $|a^{(C)}| \equiv 1 \mod 2$. Denoting $h = \max ( \Height( a^{(\epsilon)} ), \Height( b^{(\epsilon)} ) )$, $x^{(\epsilon)}$ is a solution to Equation \eqref{equation:bepsilont} if and only if
\[
    x^{(\epsilon)} = x^{(\epsilon)}_{ -h } \in [ b^{(\epsilon)} + t a^{(\epsilon)}, b^{(\epsilon)} + (t + 1) a^{(\epsilon)} ]
\]
in $V( \mathcal{L}_h )$.
\end{claim}

\begin{proof}
If $|a^{(C)}| \equiv 1 \mod 2$, then Equation \eqref{equation:bepsilont} reduces to
\[
    x^{(\epsilon)} = a^{(\epsilon)} x^{(\epsilon)} + b^{(\epsilon)} + t a^{(\epsilon)}.
\]
We note that this is not equivalent to $(a^{(\epsilon)} + 1)  x^{(\epsilon)} = b^{(\epsilon)} + t a^{(\epsilon)}$, for the multiplicative identity $1 = C_1$ does not belong to $L$. However, since the height of the right hand side is at most $h = \max ( \Height( a^{(\epsilon)} ), \Height( b^{(\epsilon)} ) )$, we have $\Height( x^{(\epsilon)} ) \le h$ and hence $x^{(\epsilon)} = x^{(\epsilon)}_{ -h }$.
The equation above is then equivalent to
\[
    ( a^{(\epsilon)} + L_h) x^{(\epsilon)} = b^{(\epsilon)} + t a^{(\epsilon)},
\]
which by Theorem \ref{theorem:division_Lh}, yields $x^{(\epsilon)} = x^{(\epsilon)}_{ -h } \in [ b^{(\epsilon)} + t a^{(\epsilon)}, b^{(\epsilon)} + (t + 1) a^{(\epsilon)} ]$.
\end{proof}

\end{proof}

Unfortunately, characterising the set $X_t$ is not straightforward, for the following reason. Let $\mathcal{P} = \{ e \in S_0 : |e| \equiv 0 \mod 2 \}$ denote the ``parity-check code.'' Then for any interval $[\lambda, \upsilon]$ with $\lambda < \upsilon$, the intersection $[\lambda, \upsilon] \cap \mathcal{P}$ contains infinite ascending chains: if $e \in [\lambda, \upsilon] \cap \mathcal{P}$ and $k = \lcm( K( \lambda ), K( \upsilon ), K( x ) )$, then $f = e + C_k + C_{3k}$ satisfies $f \in [\lambda, \upsilon] \cap \mathcal{P}$ and $f > e$. Similarly, the set $[\lambda, \upsilon] \setminus \mathcal{P}$ contains infinite descending chains instead.

Nonetheless, if we view solving equations over $\F_2$ as a simplification of solving equations over $\N$, where there are strict bounds on the possible size of a solution, then we can make some restrictions on $x$ such that $ax = b$ in $\F_2(C \cup L)$. Let $k$ be an odd integer such that $\lcm( K(a^{(C)}), K(b^{(C)}) ) \ | \ k$, i.e. $a^{(C)} = a^{(C)}_{| \ k}$ and $b^{(C)} = b^{(C)}_{| \ k}$, or in other words, all cycle lengths in $a$ and $b$ divide $k2^m$ for some $m$. Then the solutions with $x^{(C)} = x^{(C)}_{| \ k}$ can be more precisely determined, since we are now able to stay in finite Boolean algebras throughout.

For all $i \in \N$, let $\lambda_i, \upsilon_i$ be as in Theorem \ref{theorem:ax=b} for the division $a^{(C)} x^{(C)} = b^{(C)}$ and define the intervals $J(i,0) =[\lambda_i + C_k + \lambda_i C_k, \upsilon_i]$ and $J(i,1) = [ \lambda_i, \upsilon_i + \upsilon_i C_k ]$ of $V( \mathcal{C}_k )$. Then, for all $t \in \{0,1\}$, let
\begin{align*}
    W_t &= \left\{ \omega \in \{0,1\}^* : \sum_i \omega_i \equiv t \mod 2 \right\} \\
    (X_t)_{| \ k} &= \bigcup_{ \omega \in W_t } \left\{ x^{(C)}_{| \ k} \in S : x_i \in J(i,\omega_i) \ \forall i \ge 0 \right\}.
\end{align*}

\begin{proposition}
Let $a,b,x \in \F_2(C \cup L)$ such that $\lcm( K(a^{(C)}), K(b^{(C)}) , K( x^{(C)} )) \ | \ k$. Then, following the notation above, $ax = b$ if and only if $x^{(C)} \in (X_t)_{| \ k}$ for some $t \in \{0,1\}$ and $x^{(\epsilon)} \in  Y_t^{(\epsilon)}$ for all $\epsilon \in \{0,1\}$.
\end{proposition}

\begin{proof}
Corollary \ref{corollary:Pl} shows that the $x^{(C)}$ such that $b^{(C)} = a^{(C)} x^{(C)}$ with $|x_i|_\zero = \omega_i$ is $J(i, \omega_i)$, with $J(i,0) =[\lambda_i + a^\zero + \lambda_i a^\zero, \upsilon_i]$ and $J(i,1) = [ \lambda_i, \upsilon_i + \upsilon_i a^\zero ]$, where $a^\zero = \zero$ is given by $C_k$ in $V( \mathcal{C}_k )$. Therefore, $x^{(C)} \in (X_t)_{| \ k}$ for $t = |x^{(C)}| \in \{0,1\}$. The rest of the proof follows Theorem \ref{theorem:division_LC}.
\end{proof}

\section{Conclusion and future work} \label{section:conclusion}

In this paper, we introduced a general framework for semirings of formal sums. We then applied this framework to study $S = \F_2 C$, where $C$ is the set of cycles. We characterise many algebraic properties of $S$, and we solve the division problem efficiently in this ring. This is in stark contrast with the original problem of dividing sums of cycles, for which no polynomial-time algorithm is known. We then extend the study to $\F_2 (L \cup C)$, i.e. considering sums of cycles and chains, which are the injective partial transformations.

The study of partial transformations, and semigroups of formal sums thereof, has many degrees of freedom.
\begin{itemize}
	\item Choice of partial transformations. The approach used in this paper could be extended to more general classes of partial transformations. A natural candidate is the class of transformations of bounded height (the maximum length of a path from a state to a periodic state, also referred to as the pre-period), which form a subsemiring.

	\item Choice of semiring $\S$. Choosing the semiring $\S = \F_2$ yielded some strong algebraic properties and dramatic simplifications to some problems (notably the division problem of sums of cycles). Other choices for $\S$ may shed different lights on the overall structure of the semiring of partial transformations. Natural examples include the Boolean semiring (which, for the product, is equivalent to working on the so-called power monoid) and other prime fields $\F_p$.

	\item Different problems. Even when restricting to the focus $S = \F_2 C$ of this paper, many interesting questions remain open. These include: the intersection of principal ideals, the polynomial identities, solving more general univariate polynomial equations, solving multivariate equations, etc.
\end{itemize}

\section*{Acknowledgment}

We would like to thank Sara Riva for fruitful discussions about the motivation behind using partial transformations.


\end{document}